%% file: main.tex
\documentclass{article}

\usepackage{arxiv}
\usepackage[T1]{fontenc}
\usepackage[utf8]{inputenc}

\usepackage{hyperref}
\usepackage{url}
\usepackage{booktabs}
\usepackage{colortbl}
\usepackage{makecell}
\usepackage{multirow}
\renewcommand{\arraystretch}{1.2}
\usepackage{amsmath,amssymb,amsfonts,amsthm}
\usepackage{nicefrac}
\usepackage{microtype}
\usepackage{graphicx}
\usepackage{subcaption}
\usepackage[
    backend=biber,
    style=authoryear,
    natbib=true,
]{biblatex}
\usepackage{fancybox}
\cornersize{0.75}

\usepackage{physics}
\usepackage{bm}
\usepackage{siunitx}
\usepackage{todonotes}
\usepackage{xfrac}

\usepackage[frozencache]{minted}

\newtheorem{lemma}{Lemma}

\definecolor{codebg}{HTML}{F7F7F7}
\definecolor{tabgray}{HTML}{EFEFEF}
\definecolor{tabred}{HTML}{EAA1A1}

\definecolor{plotblue}{HTML}{0072B2}
\definecolor{plotorange}{HTML}{E69F00}
\definecolor{plotgreen}{HTML}{009E73}

\setminted{
  bgcolor=codebg,
  fontsize=\small,
  breaklines,
  autogobble,
  tabsize=2,
  frame=lines,
  framesep=3mm,
  linenos,
  numbersep=8pt
}

\usepackage{mathtools}

\usepackage[verbose]{newunicodechar}

\newunicodechar{š}{\v{s}}
\newunicodechar{ń}{\'{n}}

\usepackage{calc}

\usepackage{pgfplots}
\usetikzlibrary{calc,angles,quotes}

\usepackage[inline]{enumitem} 

\newcommand{\trp}{\mathsf{T}}
\DeclareMathOperator{\fl}{\textrm{fl}}
\DeclareMathOperator{\sym}{\textrm{sym}}
\newcommand{\epsmach}{\epsilon_\text{mach}}

\title{Running error bounds in finite element kernels}

\author{%
    Michal HABERA \\
    Department of Engineering\\
    University of Luxembourg\\
    \texttt{michal.habera@uni.lu} \\
    \And
    Paul T. KÜHNER \\
    Department of Engineering\\
    University of Luxembourg\\
    \texttt{paul.kuehner@uni.lu} \\
    \And
    Matteo CROCI \\
    Basque Center for Applied Mathematics\\
    Ikerbasque\\
    University of the Basque Country\\
    \texttt{mcroci@bcamath.org} \\
    \And
    Andreas ZILIAN \\
    Department of Engineering\\
    University of Luxembourg\\
    \texttt{andreas.zilian@uni.lu} \\
}

\date{September 2026}
\renewcommand{\headeright}{Preprint}
\renewcommand{\undertitle}{Preprint}

\hypersetup{
pdftitle={Running error bounds in finite element kernels},
pdfauthor={M.~Habera, P.T.~Kühner, M.~Croci, A.~Zilian},
}

\begin{document}
\maketitle

\begin{abstract}
	Rounding errors in finite element computations can lead to a complete loss of accuracy, stalled convergence, and incorrect results.
	Moreover, the effects of rounding errors accumulated within automatically generated and compiled kernels are difficult to analyze a priori.
	We present the first software framework for automated rounding error estimation within finite element kernels.
	The proposed methodology is based on an a posteriori technique called Running Error Analysis (REA), where a forward error estimate is automatically computed concurrently with the value.
	An open-source implementation is provided for the FEniCS Form Compiler (FFCx), based on a C++ backend for generating type-generic templated kernels over a custom arithmetic type that tracks both the value and its error estimate.

	We demonstrate REA on two examples.
	First, we use it to detect catastrophic cancellation in the assembly of a Neo-Hooke hyperelastic model in the small deformation regime.
	A series expansion circumvents the cancellation problem and the computed error estimates show this.
	Second, we study the assembly of the Laplace operator on a near-degenerate mesh.
	We demonstrate that our REA implementation typically incurs only 2-4x performance overhead.
	Applications of this work include robust reduced-precision computations in embedded systems, numerical debugging of new, possibly ill-conditioned or unstable PDE formulations, and guiding the design of mixed-precision kernels.
\end{abstract}


\section{Introduction}
\label{sec:introduction}

Rounding errors can limit the accuracy of finite element computations and stall the convergence of iterative solvers.
Estimating the magnitude of these errors is crucial when using lower-precision arithmetic, but the effects of rounding errors accumulated within automatically generated and compiled kernels are difficult to analyze a priori.
In this work, we present a method for automated rounding error estimation in the assembly of finite element operators.
We propose a software framework that combines \emph{code generation}, i.e., the translation of the mathematical description of Partial Differential Equations (PDEs) to a lower-level programming language \citep{kirby2006compiler,alnaes2014ufl}, with \emph{running error analysis} (REA) for a posteriori rounding error bounds, see \citet[\S~3]{wilkinson1986erroranalysis} and \citet[\S~3.3]{higham2002accuracy}.
The resulting algorithm automatically computes rounding error bounds alongside the assembled values and reveals accuracy loss due to numerical instability and ill-conditioning, with modest overhead in most of the tested cases.

Double precision was available from the dawn of the computer Finite Element Method (FEM) in the 1960s, but memory constraints encouraged single precision storage \citep{melosh1969manipulationerrors}.
The IEEE single- and double-precision arithmetic became widely supported in hardware during the late 1980s and 1990s \citep[\S~2]{higham2022mixedprecision}, making the 64-bit arithmetic increasingly practical.
As a consequence, double precision became the conservative default in many engineering finite element software packages.

This conservatism can be justified when rounding errors limit the accuracy of constitutive evaluations or the convergence of iterative solvers.
In solid mechanics, for example, evaluating finite-strain constitutive models near the undeformed state can involve cancellation, as illustrated by the Neo-Hooke example in Section~\ref{sec:examples}.
Contact algorithms \citep[Ch.~10]{wriggers2006contact} and non-smooth, multi-surface plasticity algorithms \citep{simo1988multisurface,simo1998inelasticity} also use contact gaps and yield-function values to determine which constraints are active.
Rounding errors near activation thresholds can alter these decisions.
In addition, industrial CAD geometries may lead to ill-shaped elements and poorly conditioned element-local operators \citep[\S~3]{shewchuk2002elementquality}, increasing sensitivity to rounding errors.
Higher working precision can reduce these errors and help keep them below the tolerances required by the simulation.

The development of Artificial Intelligence (AI) and Large Language Models (LLMs) is shaping the landscape of available hardware and the numerical formats it supports \citep{jouppi2017tpu,micikevicius2022fp8}.
In particular, double precision is not prioritized by AI hardware manufacturers, and developers and researchers wanting to harness the latest hardware capabilities are forced to adapt their algorithms to lower-precision computations.
The result is a wide range of reduced- and mixed-precision algorithms, which leverage the performance gains of low-precision computations while controlling rounding error accumulation.
We refer to the review articles of \citet{abdelfattah2021survey} and \citet{higham2022mixedprecision} for an overview of existing methods.
The advantage of using lower precision is that these formats favor not only the applications with high arithmetic intensity \citep{williams2009roofline}, but also applications that are memory bandwidth bound, due to improved memory traffic utilization.

Overall, the rounding error behaviour in a typical finite element industrial simulation is complex, and we only try to address a small part of it in this work.

\paragraph{Rounding error detection in numerical software}

This complexity motivates the use of automated tools to assess how rounding errors propagate through numerical software.
The estimation of rounding errors in a program is addressed by an extensive list of libraries.
Based on the technique that these tools follow, one can distinguish the following approaches.

\emph{Dynamic local-error propagation} techniques augment each floating-point operation so that it returns a pair $(\hat f, e_f)$ of the computed value and an estimate of its accumulated error, and update $e_f$ after each operation.
The local rounding error can be estimated from a floating-point model or computed using error-free transformations or higher-precision values, see subsections below.
Examples include CENA \citep{braconnier2002cena,hovland2021forwardcena}, Shaman \citep{demeure2022shaman}, EFTSanitizer \citep{chowdhary2022eftsanitizer}, Accurate Residues \citep{he2026accurateresidues}, ADAPT \citep{menon2018adapt}, CHEF\_FP \citep{singh2023cheffp}, ATOMU \citep{zou2020atomu}, and FPCC \citep{yi2024fpcc}.
Advantages are that this approach provides error information for intermediate operations, computation happens alongside values in one execution, and error-free transformations can be used to avoid expensive high-precision arithmetic.
Disadvantages are the additional arithmetic and storage overhead, potentially overestimated errors for unsigned bounds, and the requirement to implement special rules for every relevant operation.

\emph{High-precision shadow execution} methods evaluate the floating-point program twice: once for the floating-point value $\hat f$, and once for a higher-precision reference value $f_\text{ref}$.
The difference $\hat f - f_\text{ref}$ is then used as an estimate of the accumulated error.
Examples include FPDebug \citep{benz2012fpdebug}, FPSanitizer and PFPSanitizer \citep{chowdhary2020fpsanitizer,chowdhary2021pfpsanitizer}, NSan \citep{courbet2021nsan}, Herbgrind \citep{sanchezstern2018herbgrind}, and SHVAL \citep{lam2016shval}.
Advantages include a directly interpretable, signed difference between the computed value and reference, and no need for first-order error-propagation approximations.
On the other hand, high-precision arithmetic and shadow storage can be expensive, in particular when precision is emulated.
The reference still has rounding errors and can itself suffer numerical instability.

\emph{Stochastic methods} repeatedly evaluate the program while randomly perturbing rounding decisions or mantissas.
The outcomes are used to estimate variance, confidence intervals, or numerical instabilities.
These methods are implemented in CADNA \citep{jezequel2008cadna}, Verificarlo \citep{denis2016verificarlo}, VERROU \citep{fevotte2016verrou}, and SAM \citep{graillat2011sam}.
These approaches provide additional statistical information such as variability, significant digits, and confidence intervals.
Unfortunately, conclusions depend on the perturbation model and statistical assumptions, and multiple samples add cost.

For a more extensive list of libraries see \citet{fptalks2026community}.
The same page lists many additional libraries classified as \emph{static}.
These tools perform rounding error analysis prior to the execution of the code, so they are more similar to an a priori error analysis and fall outside the scope of the current work.
Their main advantage is the possibility of proving bounds over a specified input domain before execution.

Among tools based on dynamic local-error propagation, Shaman \citep{demeure2022shaman} provides debugging features such as attribution of errors to their sources and detection of unstable branches, together with broad integration into C++ software.
Instead, our framework specifically focuses on finite element assembly through generated FEniCSx kernels and Python/NumPy, with support for half precision and both approximate worst-case bounds and signed error estimates.

\paragraph{Rounding errors in the finite element method}

Existing literature on the effects of rounding errors in FE kernels is scarce.
The first studies were conducted by \citet{melosh1969manipulationerrors,melosh1973inheritederror,utku1984solutionerrors}.
In these studies the focus is on the entire process of obtaining an FE solution, including the conditioning and factorization of the structural system.
Moreover, the problems are limited to structural systems with stiffness and mass matrices.

On the other hand, there is a larger amount of literature on rounding errors in the solution of linear systems arising from FEM \citep{babuska1981pversion,fried1986roundoff,alvarez2012roundoff,babuska2018roundoff,liu2021balancing}.
Some specific areas that are part of the FE kernel evaluation have also received attention.
For example, the stable and efficient way of interpolating polynomials using barycentric interpolation is known, see \citet{higham2004barycentric,berrut2004interpolation,laughton2022barycentric}.
Stability and interpolation accuracy, especially for high-order polynomials, are also studied in \citet{brubeck2025fiat}.

Perhaps the only reference that studies rounding errors in FE kernels, including reference basis evaluation, geometry evaluation, construction of gradient transformations, quadrature accumulation, and local-to-global assembly, is \citet{croci2024kernels}.
However, \citet{croci2024kernels} provide theoretical a priori rounding error bounds while the current work focuses on practical a posteriori rounding error estimates.

\paragraph{Main contributions}

The main contributions of this paper are as follows.
\begin{enumerate}
	\item We present the first software framework for automated a posteriori rounding error estimation of FE computations.
	\item We implement a custom C++ data type, \texttt{running\_error\_t}, and integrate it with templated kernels generated by the FEniCS Form Compiler (FFCx) and the FEniCSx assemblers.
	      The type computes rounding error estimates alongside the values in two modes: ``worst'' for running error bounds, and ``exact'' for tighter signed error estimates based on higher-precision arithmetic and error-free transformations.
	      In the Neo-Hooke benchmarks, the slowdown is at most $3.5\times$ relative to plain assembly in most tested cases, making the estimates practical for diagnosing rounding errors in FE kernels.
	\item We demonstrate that both error modes detect the numerical instability of a na\"ive Neo-Hooke strain energy evaluation in the small-deformation regime, despite the bounded condition number of the underlying problem.
	      Both modes also capture the improved accuracy obtained with a stable implementation using a series expansion.
	\item We show that the error estimates capture the growth of rounding errors when assembling a Laplace operator on a mesh with near-degenerate, needle-like triangles.
	      The estimates reflect the asymptotic scaling predicted by conditioning arguments: the ``worst'' mode captures worst-case growth, while the ``exact'' mode also captures milder growth arising from cancellations between roundoffs.
\end{enumerate}

\paragraph{Paper outline}
Section~\ref{sec:rounding-error-analysis} introduces the floating-point model and notation and derives running error bounds and alternative signed error estimates.
Section~\ref{sec:implementation} describes their implementation in automatically generated FE kernels.
Section~\ref{sec:examples} presents the Neo-Hooke and near-degenerate-mesh examples, including conditioning analyses, comparisons of the error estimates, and performance measurements.

\section{A posteriori rounding error analysis}
\label{sec:rounding-error-analysis}

We review two approaches that estimate rounding errors alongside computed values: first-order running error bounds and signed error estimates.
Both separate the error introduced by the current operation from errors propagated through its inputs.

\subsection{Standard floating-point model and notation}

In this work we follow the notation of \citet{higham2002accuracy}.
Let $\fl(x)$ be a \emph{faithful} rounding function, returning $x$ if it is exactly representable and otherwise either of the two adjacent floating-point numbers that bracket $x$.
The IEEE standard model of floating-point arithmetic, see \citet[Eq.~2.4]{higham2002accuracy}, guarantees on most modern hardware, excluding underflow and overflow, that computing an operation $x \textrm{ op } y$ incurs a relative error no worse than \emph{machine epsilon} $\epsmach$, i.e.,
\begin{equation}
	\fl(x \text{ op } y) = (x \text{ op } y)(1 + \delta), \quad |\delta| \leq \epsmach, \quad \text{op} \in \{+, -, *, /\}.
	\label{eq:std-model}
\end{equation}
The same applies to the representation of a number in the floating-point system,
\begin{equation}
	\fl(x) = x (1 + \delta), \quad |\delta| \leq \epsmach.
	\label{eq:std-model-repr}
\end{equation}
The \emph{machine epsilon} is a property of a floating-point system defined as $\epsmach \coloneqq \beta^{1-t}$, where $\beta$ is the \emph{base} and $t$ is called \emph{precision}.
The machine epsilon is related to \emph{unit roundoff} $u \coloneqq \frac{1}{2} \beta^{1-t} = \frac{1}{2} \epsmach$.
For instance, the IEEE 64-bit double precision format has $\beta = 2$ and $t = 53$, so $\epsmach^{(64)} = 2^{-52} \approx 2.22 \times 10^{-16}$.
For the IEEE-754 binary formats we will use symbols $\epsmach^{(16)} \approx 9.77 \times 10^{-4}$, $\epsmach^{(32)} \approx 1.19 \times 10^{-7}$, and $\epsmach^{(64)} \approx 2.22 \times 10^{-16}$ for the machine epsilon of the 16-, 32-, and 64-bit precision formats, respectively.

Another useful variant of the IEEE standard model is from \citet[Eq.~(2.5)]{higham2002accuracy},
\begin{equation}
	\begin{aligned}
		\fl(x \text{ op } y) = \frac{x \text{ op } y}{1 + \delta}, \quad \fl(x) = \frac{x}{1 + \delta}, \quad |\delta| \leq \epsmach,
	\end{aligned}
	\label{eq:std-model2}
\end{equation}
which is implied by the model \eqref{eq:std-model} whenever $x \text{ op } y$ lies in the range of the floating-point system.
Both models \eqref{eq:std-model} and \eqref{eq:std-model2} are useful both for a priori and a posteriori rounding error analysis.

For mathematical library functions such as sine, cosine, and logarithm, a relative error bound of $\epsmach$ is not generally guaranteed.
The accuracy depends on the math library and its implementation.

\paragraph{Notation}

We indicate with $C > 0$ a generic constant.
It may take a different value at each occurrence, and is consistent only within a single equation.
We say that $f(x)$ and $g(x)$ are asymptotically equivalent for $x \to 0$ and write $f(x) \sim C g(x)$ if $\lim_{x \to 0} f(x) / g(x) = C$.
The $f \approx g$ means that $f$ is approximately $g$, without specifying a precise sense of approximation.
If there exists $C > 0$ such that $|f(x)| \le C |g(x)|$ for all $x$ sufficiently close to zero we use the Big-O symbol $f(x) = \mathcal O(g(x))$.
Finally, a bare $\delta$ always denotes a relative rounding perturbation, as in Eq.~\eqref{eq:std-model}, while $\delta$ applied to a field, such as $\delta \bm u$ or $\delta W$, denotes its first variation.

\subsection{Running error analysis}

Running Error Analysis (REA) is a technique for computing a posteriori bounds on forward rounding errors.
Let $x$ and $y$ denote the exact inputs to an operation and $\hat x$ and $\hat y$ their floating-point approximations.
We distinguish three results:
\begin{equation*}
	\begin{aligned}
		f        & \coloneqq x \text{ op } y,                &  & \text{(exact result for exact inputs)},     \\
		\tilde f & \coloneqq \hat x \text{ op } \hat y,      &  & \text{(exact result for perturbed inputs)}, \\
		\hat f   & \coloneqq \fl(\hat x \text{ op } \hat y), &  & \text{(computed result)}.
	\end{aligned}
\end{equation*}
The forward error separates into the local rounding error introduced by the current operation and the error propagated from its inputs:
\begin{equation}
	\hat f - f = \underbrace{(\hat f - \tilde f)}_{\text{local rounding error}} + \underbrace{(\tilde f - f)}_{\text{propagated error}}.
	\label{eq:error-split}
\end{equation}

For the local contribution, Eq.~\eqref{eq:std-model2} gives $\hat f = \tilde f / (1 + \delta)$ with $|\delta| \leq \epsmach$.
Hence,
\begin{equation}
	|\hat f - \tilde f| = |\delta| \, |\hat f| \leq \epsmach |\hat f|.
	\label{eq:local-bound-rae}
\end{equation}
This bound is computable from the result of the operation.

The propagated contribution depends on the input errors.
If these satisfy
\begin{equation}
	e_x = \hat x - x, |e_x| \le \bar e_x, \quad e_y = \hat y - y, |e_y| \leq \bar e_y,
\end{equation}
we estimate the propagated error using a first-order Taylor expansion about the exact input pair $(x,y)$, with increments $(e_x,e_y)$:
\begin{equation}
	\tilde f = f + \pdv{f}{x} e_x + \pdv{f}{y} e_y + \text{h.o.t.}.
	\label{eq:rea-expansion}
\end{equation}
Therefore we obtain the estimated bound for the propagated error
\begin{equation}
	|\tilde f - f| \leq \left| \pdv{f}{x} \right| \bar e_x + \left| \pdv{f}{y} \right| \bar e_y + \text{h.o.t.}.
\end{equation}
Combining this estimate with the local bound in Eq.~\eqref{eq:local-bound-rae} and applying the triangle inequality to Eq.~\eqref{eq:error-split} gives
\begin{equation}
	|\hat f - f| \leq \epsmach |\hat f| + \left| \pdv{f}{x} \right| \bar e_x + \left| \pdv{f}{y} \right| \bar e_y + \text{h.o.t.} \qquad \text{(REA error bound)}.
	\label{eq:rea-general}
\end{equation}
This is an approximate, worst-case upper bound to first-order in $\epsmach$, $\bar e_x$ and $\bar e_y$.
The resulting error estimate is stored alongside the computed value and propagated as an input-error estimate in subsequent operations.
To obtain a computable estimate, we evaluate the derivatives at $(\hat x,\hat y)$ instead.
For a twice continuously differentiable operation, this replacement changes the propagated-error estimate only by terms of second order in the input errors, which are absorbed into the higher-order terms.
All derivatives in the subsequent error estimates are evaluated at $(\hat x,\hat y)$.

\begin{table}[ht]
	\centering
	\renewcommand{\arraystretch}{1.5}
	\begin{tabular}{>{\columncolor{tabgray}}l||l}
		\rowcolor{tabgray}
		Operation              & Running error bound                                                                        \\ \hline\hline
		$\fl(\hat x + \hat y)$ & $\epsmach |\hat f| + \bar e_x + \bar e_y + \text{h.o.t.}$                                  \\ \hline
		$\fl(\hat x \hat y)$   & $\epsmach |\hat f| + |\hat y| \bar e_x + |\hat x| \bar e_y + \text{h.o.t.}$                \\ \hline
		$\fl(\hat x / \hat y)$ & $\epsmach |\hat f| + (|\hat y| \bar e_x + |\hat x| \bar e_y) / |\hat y|^2 + \text{h.o.t.}$ \\
	\end{tabular}
	\vspace{1em}
	\caption{Running error bounds for a few elementary operations.}
	\label{tab:bounds-elem}
\end{table}

Running error bounds for a few elementary operations are included in Tab.~\ref{tab:bounds-elem}.
They can be found in \citet[Tab.~1]{zahradnicky2010running}.

\subsection{Signed error estimates}

Taking absolute values in the running bound prevents cancellation between error contributions.
Signed error estimates retain this information.
The following discussion reviews their construction using first-order propagation, higher-precision evaluation, and error-free transformations.
Computing local rounding errors exactly and propagating input errors through a first-order expansion is related to the pair arithmetic developed in \citet{lange2020faithfully}.

Let $f_\text{ref} \approx \tilde f = \hat x \text{ op } \hat y$ be a reference value computed in a sufficiently high-precision or arbitrary-precision floating-point format.
Retaining the signs of the propagated errors and combining the first-order expansion in Eq.~\eqref{eq:rea-expansion} with this reference value gives the signed error estimate
\begin{equation}
	\hat f - f \approx \underbrace{(\hat f - f_\text{ref})}_{\text{estimated local rounding error}} + \underbrace{\pdv{f}{x} e_x +\pdv{f}{y} e_y}_{\text{propagated error}} \qquad \text{(Signed error estimate)}.
	\label{eq:rea-higher-precision}
\end{equation}

Another option is to use error-free transformations and compute the local rounding error $\hat f - \tilde f$ exactly.
An error-free transformation computes the floating-point result $s = \fl(x \text{ op } y)$ and a residual $r$ such that $x \text{ op } y = s + r$.
Applied to the perturbed inputs, it gives $\tilde f = \hat f + r$, so the local forward error is $\hat f - \tilde f = -r$.
For the operation of addition, \texttt{TwoSum} or \texttt{FastTwoSum} \citep{dekker1971,ogita2005} are the most commonly used error-free transformations.
An error-free transformation for the multiplication could be achieved with \texttt{TwoProductFMA} \citep[Alg.~3.5]{ogita2005}, if the fused multiply-add (FMA) instruction is available, or with \texttt{TwoProduct} at the cost of more floating-point operations \citep[Alg.~3.3]{ogita2005}.

The error estimates based on Eq.~\eqref{eq:rea-higher-precision} are signed error estimates, unlike the nonnegative bound in Eq.~\eqref{eq:rea-general}.
While neither of them can be used as a certified bound, as, e.g., in interval arithmetic \citep{moore2009interval}, a large value of the estimated error might suggest that
\begin{enumerate*}[label=\arabic*)]
	\item the problem we are dealing with is ill-conditioned, or
	\item the numerical computation is unstable.
\end{enumerate*}

We remark that while the models Eq.~\eqref{eq:std-model} and Eq.~\eqref{eq:std-model2} provide accurate upper bounds for rounding errors, the best-case lower bound is always zero since floating-point computations can be exact for specific input values, see \citet[\S 4.2]{muller2018handbook}.
Automatically computing more informative lower bounds is, in general, nontrivial.

In this work, we implement this signed-error approach alongside running error bounds in a common arithmetic type and integrate both modes into automatically generated FEniCSx kernels and finite element assembly.
This integration enables their accuracy and computational cost to be compared within the same assembly workflow.

\section{Implementation for automatically generated FE kernels}
\label{sec:implementation}

The implementation of the running error bounds for the numerical experiments in this paper was done in the REA library \citep{habera2026rea} in combination with the dolfiny package \citep{zilian2026dolfiny}, which is based on the FEniCSx finite element ecosystem \citep{baratta2025dolfinx}.
FEniCSx assembles finite element forms representing weak formulations of PDEs into scalars, vectors, or matrices.
The domain-specific language in which the PDE is defined is the Unified Form Language (UFL), see \citet{alnaes2014ufl}.
FFCx \citep{fenics2026ffcx} translates the symbolic, high-level PDE description into a lower-level programming language.

\paragraph{Custom backend}

Firstly, we have developed a custom backend for the form compiler, see \citet{kuhner2026ffcxbackends}, which allows us to generate templated C++ code.
The templated code is necessary for us in order to compile the kernel with a custom data type that overloads arithmetic operations to track the running error.
The signature of a generated FFCx kernel is shown in Fig.~\ref{fig:ffcx-kernel}.

\begin{figure}[ht]
	\begin{minted}{cpp}
template <typename T, typename U>
void tabulate_tensor(T* __restrict A,
                     const T* __restrict w,
                     const T* __restrict c,
                     const U* __restrict coordinate_dofs,
                     const std::int32_t* __restrict entity_local_index,
                     const std::uint8_t* __restrict quadrature_permutation);
\end{minted}
	\caption{Signature of a generated FFCx kernel.}
	\label{fig:ffcx-kernel}
\end{figure}

The body of the kernel follows the typical structure described in \citet{baratta2025dolfinx}.
It produces arrays of basis functions evaluated at quadrature points, evaluates the geometry information (Jacobian determinant and other transformations), builds the evaluation tree for the expressions representing the integrands, and finally loops over quadrature points and free indices of the accumulation tensor \texttt{A}, which is the only output of the kernel.

\paragraph{Custom C++ type}

Secondly, a custom C++ type called \texttt{running\_error\_t<T, Mode>} implements various running error bounds discussed above.
This type is provided by the REA library, which is distributed as a single header file and requires C++23 \citet{habera2026rea}.
\footnote{We require C++23 for the use of fixed-width, floating-point types: \texttt{std::float16\_t}, \texttt{std::float32\_t} and \texttt{std::float64\_t}.}
Its data layout is shown in Fig.~\ref{fig:re_t-layout}.
The first $N$ bits store the value of the number, followed by the error (nonnegative bound or signed estimate) stored in the subsequent $N$ bits.
The total size of the structure is $2N$ bits.

\begin{figure}[ht]
	\centering
	\begin{minipage}[c]{0.58\textwidth}
		\begin{minted}{cpp}
template <typename T, ErrorMode Mode = ErrorMode::WORST>
struct running_error_t {
	T val;
	T err;
};
        \end{minted}
	\end{minipage}
	\hfill
	\begin{minipage}[c]{0.38\textwidth}
		\centering
		\begin{tikzpicture}[x=1.28cm, y=0.55cm]
			\draw[fill=blue!45] (0,0) rectangle (2,1);
			\draw[fill=blue!20] (2,0) rectangle (4,1);

			\node[font=\scriptsize\ttfamily] at (1,0.5) {T val};
			\node[font=\scriptsize\ttfamily] at (3,0.5) {T err};

			\node[font=\scriptsize] at (1,1.45) {$N$ bits};
			\node[font=\scriptsize] at (3,1.45) {$N$ bits};

			\draw[<->] (0,1.15) -- (2,1.15);
			\draw[<->] (2,1.15) -- (4,1.15);

			\node[font=\scriptsize] at (2,-0.45) {$2N$ bits};
			\draw[<->] (0,-0.15) -- (4,-0.15);
		\end{tikzpicture}
	\end{minipage}

	\caption{Layout of the custom running error type, \texttt{running\_error\_t}.
		The memory layout for \texttt{running\_error\_t<T>} with an $N$-bit type \texttt{T} is shown on the right.
	}
	\label{fig:re_t-layout}
\end{figure}
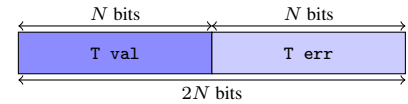

The optional \texttt{ErrorMode} takes values:
\begin{itemize}
	\item \texttt{WORST} for the REA bounds from Eq.~\eqref{eq:rea-general}.
	      We will refer to this mode as ``worst''.
	\item \texttt{EXACT} for the higher precision error estimate from Eq.~\eqref{eq:rea-higher-precision}.
	      This mode will be referred to as ``exact''.
\end{itemize}

In the ``exact'' mode, we use error-free transformations for addition and subtraction (\texttt{TwoSum}) and multiplication (\texttt{TwoProductFMA}).
For division, we use FMA-based error estimation, following \texttt{CPairDiv} in \citet{lange2020faithfully}.
Although the FMA residual can be computed exactly under appropriate assumptions, the quotient correction obtained from it is subject to rounding.
The current implementation does not fully account for subnormal values and underflow in multiplication and division, which can invalidate the exactness of the computed residuals.
Suitable scaling and squeezing techniques could mitigate these range limitations.

For mathematical functions such as sine, cosine, and logarithm, the ``worst'' mode uses $\epsmach$ as the assumed local relative error bound, although the accuracy of the math library may not satisfy this assumption.
In the ``exact'' mode, we compute a reference value $f_\text{ref}$ in a wider floating-point format.
We use the format of twice the width of \texttt{T}, if available.
For example, on many \texttt{aarch64} platforms, the wider formats for 16-, 32-, and 64-bit formats are available in the form of 32-, 64-, and 128-bit IEEE-754 formats, see \citet[Sec.~10.1.1, Tab.~3]{aapcs64}.
However, on many x86-64 platforms there is only the extended double precision \texttt{long double} format with 80-bit width and 64-bit significand precision.

All basic arithmetic operators for \texttt{running\_error\_t} are overloaded to follow the semantics of the respective error modes.
For example, in Fig.~\ref{fig:cpp-overload} the \texttt{operator*} between two running error numbers is overloaded to compute the value in the precision of type \texttt{T}, and to compute the error bound for the ``worst'' mode, according to Eq.~\eqref{eq:rea-general}.

\begin{figure}[ht]
	\begin{minted}{cpp}
  running_error_t operator*(const running_error_t& other) const {
    const T new_val = val * other.val;
    return running_error_t{
        new_val,                                                            // value
        abs(other.val) * err + abs(val) * other.err + eps * abs(new_val)};  // error
  }
\end{minted}
	\caption{Example overloading of the multiplication operator between two running error numbers for the ``worst'' mode.}
	\label{fig:cpp-overload}
\end{figure}

Operator overloading for the local-error propagation is not a novel technique, and has been used in the CENA prototype \citep{hovland2021forwardcena}, and the Shaman library \citep{demeure2022shaman}, while the concepts were already presented in \citet{zahradnicky2010running}.
We could have also used any other publicly available software implementation reviewed in Section~\ref{sec:introduction}.
Instead, we have decided to implement the custom \texttt{running\_error\_t} because:
\begin{enumerate*}[label=\arabic*)]
	\item the implementation is simple, with roughly 1000 lines of C++23 header-file code,
	\item we use different error modes to compare the accuracy of different estimates,
	\item we require full control over the binary layout of the type, in order to interface C++ and Python/NumPy,
	\item we require support for the 16-bit floating-point type \texttt{std::float16\_t}.
\end{enumerate*}

FEniCSx is implemented in C++, but its Python interface is generated using nanobind \citep{nanobind}, so we require a C++ type that is compatible with both the constraints on the template type in FEniCSx and the conversion between nanobind and NumPy arrays.
Nanobind's \texttt{ndarray} is based on the DLPack array exchange protocol, which unfortunately causes it to be more restrictive.
There is no support for conversion of a structured C++ type \texttt{running\_error\_t} into a structured NumPy array.
We bypass this limitation by exporting the \texttt{running\_error\_t} as a NumPy scalar type of the same width (e.g., \texttt{complex128} for \texttt{running\_error\_t<std::float64\_t>}, or \texttt{float32} for \texttt{running\_error\_t<std::float16\_t>}), and providing a view of the value and the error.

Overall, the implementation can be summarized in the following steps:
\begin{enumerate}
	\item FFCx and ffcx-backends generate a templated C++ FE kernel matching the signature in Fig.~\ref{fig:ffcx-kernel},
	\item dolfiny compiles the kernel with \texttt{running\_error\_t<T, Mode>} from rea, using CppJIT's Just-In-Time compilation \citep{cppjit},
	\item dolfiny compiles the dolfinx assembly loops and other FE constructs for \texttt{running\_error\_t<T, Mode>},
	\item dolfinx assembles into scalars, vectors, or matrices of the type \texttt{running\_error\_t<T, Mode>}.
\end{enumerate}
By the end of the process we have two assembled quantities.
For a vector assembly, this returns a vector of values assembled in precision \texttt{T} and a vector of error bounds/estimates.
Both returned vectors have the same dimensions.

\section{Examples}
\label{sec:examples}

We propose two examples aimed to demonstrate the accuracy and performance of the running error bounds.
First, we consider the evaluation of the hyperelastic Neo-Hooke energy.
We show how the developed framework detects numerical instabilities in its na\"ive implementation.

Second, we consider the assembly of a Laplace operator on a near-degenerate mesh.
Such a problem is ill-conditioned \citep{croci2024kernels} and our running error estimates correctly capture this property.

\subsection{Hyperelastic energy of the Neo-Hooke model}

The effects of rounding errors in the evaluation of finite-strain hyperelastic models were recently studied by \citet{shakeri2024stabilenumerics}.
We illustrate the problem on a version of the Neo-Hooke hyperelastic energy, namely the variant $W_a$ in \citet[Eq.~2.11]{pence2014neohookean}.
Under a 2D plane-strain kinematic simplification the strain energy density is given by
\begin{equation}
	W = \frac{\mu}{2} (I_1 - 2 - 2 \log J) + \frac{\lambda}{2}(J - 1)^2 \qquad \text{(unstable)},
	\label{eq:energy-unstable}
\end{equation}
where $\mu$ is the shear modulus, $\lambda$ is the first Lam\'{e} parameter of the underlying 3D model, and the trace of the right Cauchy--Green tensor $\bm C \coloneqq \bm F^\trp \bm F \in \mathbb R^{2 \times 2}$ is denoted $I_1 \coloneqq \tr(\bm C)$.
The deformation gradient $\bm F \coloneqq \bm I + \nabla \bm u \in \mathbb R^{2 \times 2}$ is computed from the displacement $\bm u$, and the volumetric deformation is measured by $J \coloneqq \det(\bm F)$.
In the above, both tensors $\bm C$ and $\bm F$ are two-dimensional, since they were reduced under the plane-strain assumption.
We followed the plane-strain reduction for simpler visualization below, but all arguments and observations carry over to the original 3D problem.

The key observations in \citet{shakeri2024stabilenumerics} are that the small increment in the deformation gradient $\bm F$ loses $k$ digits of accuracy for displacement gradients of order $\|\nabla \bm u\| \approx 10^{-k}$, since the displacement gradient is added to the identity tensor.
The constant contribution of the identity is later subtracted, either in $J - 1$ or in $\log J$, but the lost digits that represent the effective strain energy contribution cannot be recovered.
This is important in particular for small deformations, and it is a typical example of catastrophic cancellation.
The authors present a solution based on the reformulation using the Green--Lagrange strain tensor $\bm E \coloneqq \frac{1}{2} (\bm C - \bm I)$, which has the property
\begin{equation}
	\bm E(\bm u) = \frac{1}{2} (\nabla \bm u + \nabla \bm u^\trp) + \mathcal O(\|\nabla \bm u\|^2),
\end{equation}
i.e., it has no constant contributions for small deformations.
In addition to that, specialized implementations of the function $\log(1 + x)$ using C99 \texttt{log1p} must be used to preserve the numerical stability.

An alternative approach is presented in \citet{habera2026automated}.
It is valid only in the regime of small deformations, since it is based on the series expansion of the strain energy in a small dimensionless parameter that measures the magnitude of the deformation, proportional to $\|\nabla \bm u\|$.
The strain energy density expansion reads
\begin{equation}
	W_\text{stable} = W^{(2)} + W^{(3)} + \mathcal O(\|\nabla \bm u\|^4) \qquad \text{(stable)},
	\label{eq:energy-expansion}
\end{equation}
with the second- and third-order contributions in $\|\nabla \bm u\|$
\begin{equation}
	\begin{aligned}
		W^{(2)} & \coloneqq \mu \tr(\bm E_1^2) + \frac{\lambda}{2} \tr(\bm E_1)^2, \\ W^{(3)} & \coloneqq \mu \Big( 2 \bm E_1 \colon \bm E_2 - \frac{4}{3} \tr(\bm E_1^3) \Big) + \lambda \Big(\tr(\bm E_1) \tr(\bm E_2) + \frac{1}{2} \tr(\bm E_1)^3 - \tr(\bm E_1) \tr(\bm E_1^2) \Big),
	\end{aligned}
\end{equation}
where the split of the Green--Lagrange strain tensor is $\bm E = \bm E_1 + \bm E_2$ with linear $\bm E_1 \coloneqq \frac{1}{2} (\nabla \bm u + \nabla \bm u^\trp)$ and non-linear $\bm E_2 \coloneqq \frac{1}{2} \nabla \bm u^\trp \nabla \bm u$ strain contributions.

\paragraph{Problem setup}
To demonstrate the catastrophic cancellation in the energy expression in Eq.~\eqref{eq:energy-unstable}, and the behaviour of our running error bounds, we assemble two linear forms
\begin{equation}
	\begin{aligned}
		L(v; \bm u)               \coloneqq \int_\Omega W(\bm u) |\mathcal K|^{-1} v \, \mathrm dx, \qquad
		L_\text{stable}(v; \bm u) \coloneqq \int_\Omega W_\text{stable}(\bm u) |\mathcal K|^{-1} v \, \mathrm dx,
	\end{aligned}
\end{equation}
for test functions $v \in W_h$, where $W_h$ is a space of cell-wise constant functions.
The inverse cell volume $|\mathcal K|^{-1}$ can be understood as an inverse Riesz map, so the assembled vector has the meaning of a cell-averaged strain energy density function.
The forms $L$ and $L_\text{stable}$ are assembled for a provided displacement field $\bm u \in U_h$, where $U_h$ is a vector-valued continuous Lagrange space of degree $k = 1$.
We denote the assembled vectors as $b_i \coloneqq L(\varphi_i; \bm u)$ where $\varphi_i$ is the $i$-th Lagrange basis function.
The displacement field $\bm u$ is computed in a pre-processing step as a solution to a cantilever deformation problem shown in Fig.~\ref{fig:cantilever-neo-hooke}, using a simpler St.~Venant--Kirchhoff material law, and a prescribed traction load $t_y$.
The input errors associated with the displacement degrees of freedom are initialized randomly at the scale of machine precision, using a fixed seed.

\begin{figure}[ht]
	\centering
	\begin{tikzpicture}[scale=1.3, >=stealth]
		\draw[thick] (0,0) rectangle (5,1);

		\foreach \y in {0,0.1,...,1.1}
		\draw (0,\y) -- (-0.12,\y-0.1);

		\foreach \y in {1.0,0.8,0.6,0.4,0.2}
		\draw[->,thick] (5.12,\y) -- (5.12,\y-0.18);
		\node at (5.12,1.2) {$t_y$};

		\draw[<->] (0,-0.3) -- (5,-0.3)
		node[midway,below] {$w = 0.5\,\mathrm{m}$};

		\draw[<->] (5.55,0) -- (5.55,1)
		node[midway,right] {$h = 0.1\,\mathrm{m}$};

		\draw[->,thick] (0,0) -- (0.8,0)
		node[below right] {$x$};
		\draw[->,thick] (0,0) -- (0,0.7)
		node[above left] {$y$};
	\end{tikzpicture}
	\caption{Cantilever deformation problem for the Neo-Hooke cancellation example.}
	\label{fig:cantilever-neo-hooke}
\end{figure}
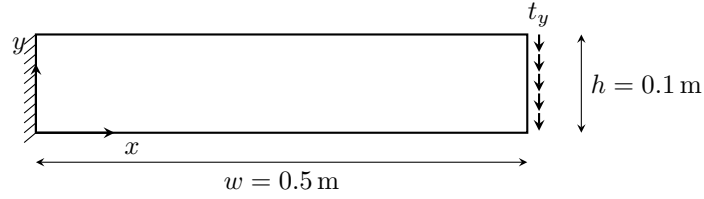

Material parameters were chosen to be representative of steel, with shear modulus $\mu \approx \SI{77}{\giga\pascal}$ and first Lam\'{e} parameter $\lambda \approx \SI{115}{\giga\pascal}$, which follows from a Young's modulus of \SI{200}{\giga\pascal} and Poisson's ratio $\nu = 0.3$.
The applied traction is set to $t_y = \alpha \cdot \si{\mega\pascal}$.

\paragraph{Spatial distribution of the error}
In this first example, we assemble the two linear forms into vectors using the 32-bit single precision floating-point format, i.e., we use \texttt{running\_error\_t<std::float32\_t>}, while the mesh geometry, and geometric quantities are stored and computed in 64-bit double precision numbers.
The load scale is set to $\alpha = 1$, which results in the deformation scale $\|\nabla \bm u\|$ varying between $10^{-5}$ and $10^{-3}$.

\begin{figure}[ht]
	\centering
	\begin{subfigure}{0.48\textwidth}
		\centering
		\includegraphics[width=\textwidth]{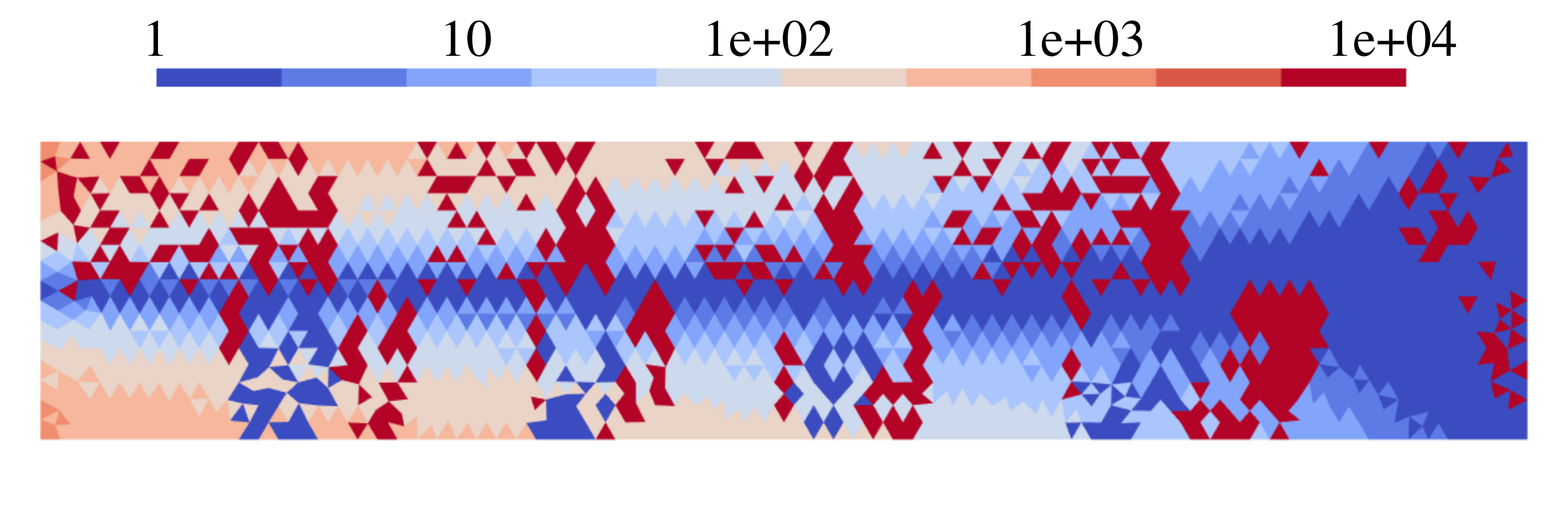}
		\caption{Assembled unstable strain energy density $W$ (in $\si{\pascal}$).}
	\end{subfigure}\hfill
	\begin{subfigure}{0.48\textwidth}
		\centering
		\includegraphics[width=\textwidth]{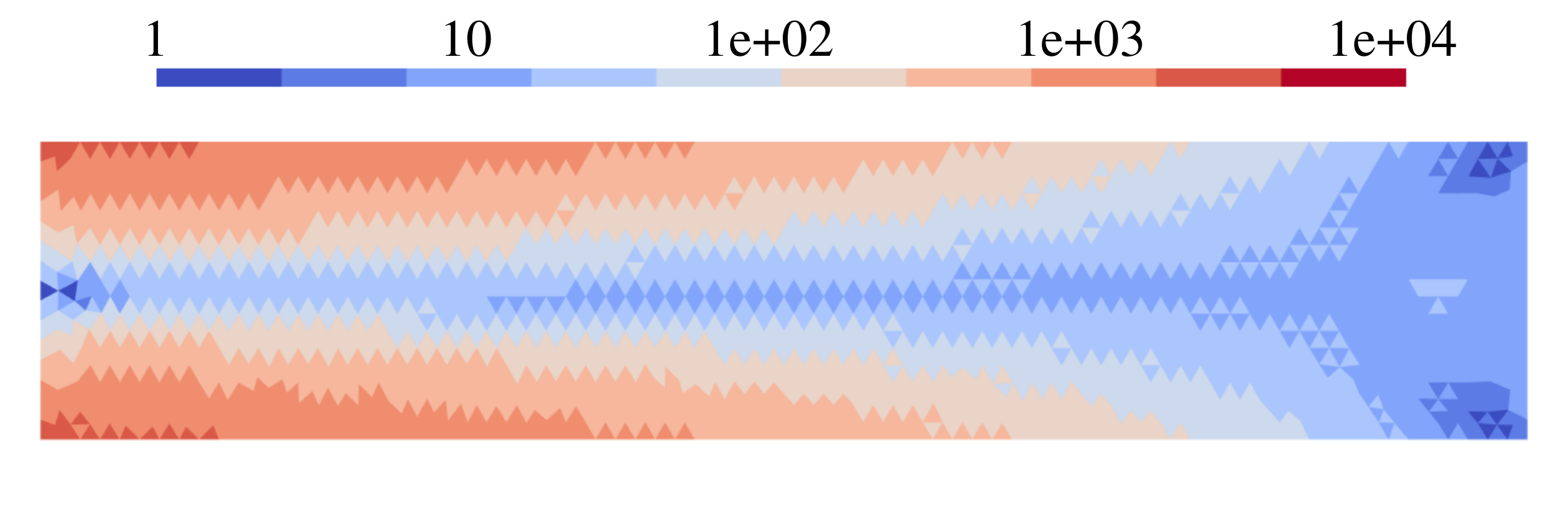}
		\caption{Assembled stable strain energy density $W_\text{stable}$ (in $\si{\pascal}$).}
	\end{subfigure}
	\caption{Comparison of the unstable and stable expressions for the strain energy density.
		Figures use a logarithmic colormap.
	}
	\label{fig:neo-hooke-energies}
\end{figure}

Results for the two energy densities $W$ and $W_\text{stable}$ are shown in Fig.~\ref{fig:neo-hooke-energies}.
The unstable strain energy density from Eq.~\eqref{eq:energy-unstable} shows cells with a complete loss of accuracy.
Additionally, we observe cells with negative strain energy densities.
On the other hand, the more stable series expansion in Eq.~\eqref{eq:energy-expansion} leads to milder rounding errors throughout the mesh.

\begin{figure}[ht]
	\centering
	\begin{subfigure}[b]{0.48\textwidth}
		\centering
		\includegraphics[width=\textwidth]{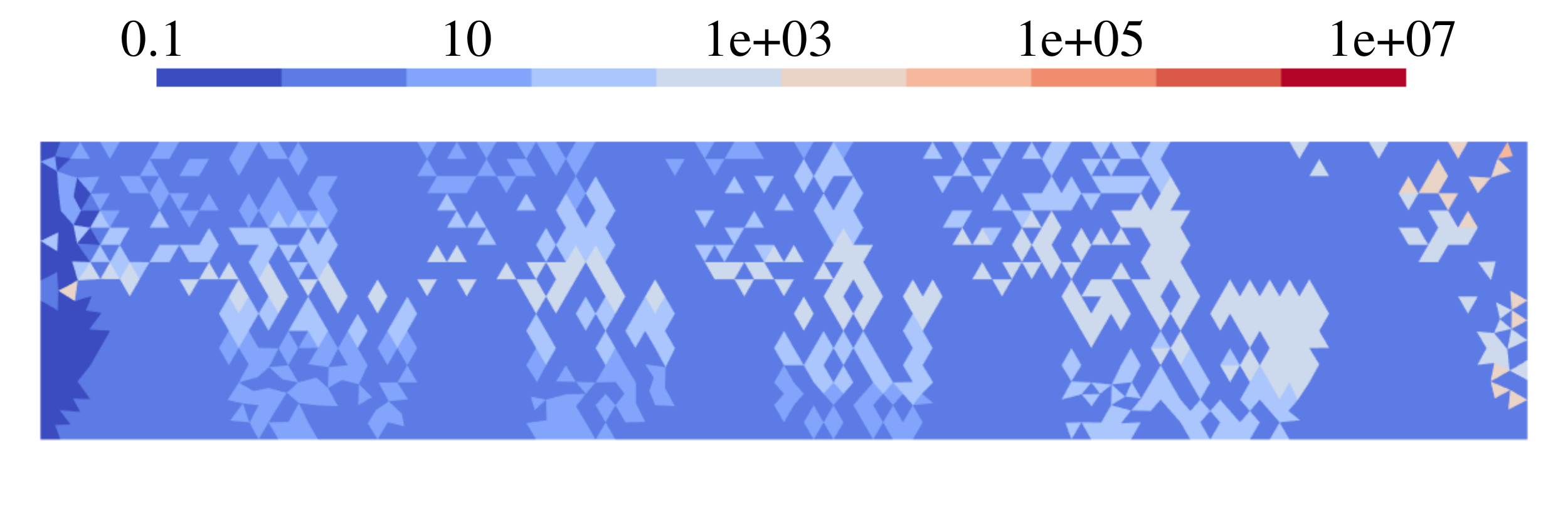}
		\caption{Rel. error estimate for $W$ and ``exact'' mode.}
	\end{subfigure}
	\hfill
	\begin{subfigure}[b]{0.48\textwidth}
		\centering
		\includegraphics[width=\textwidth]{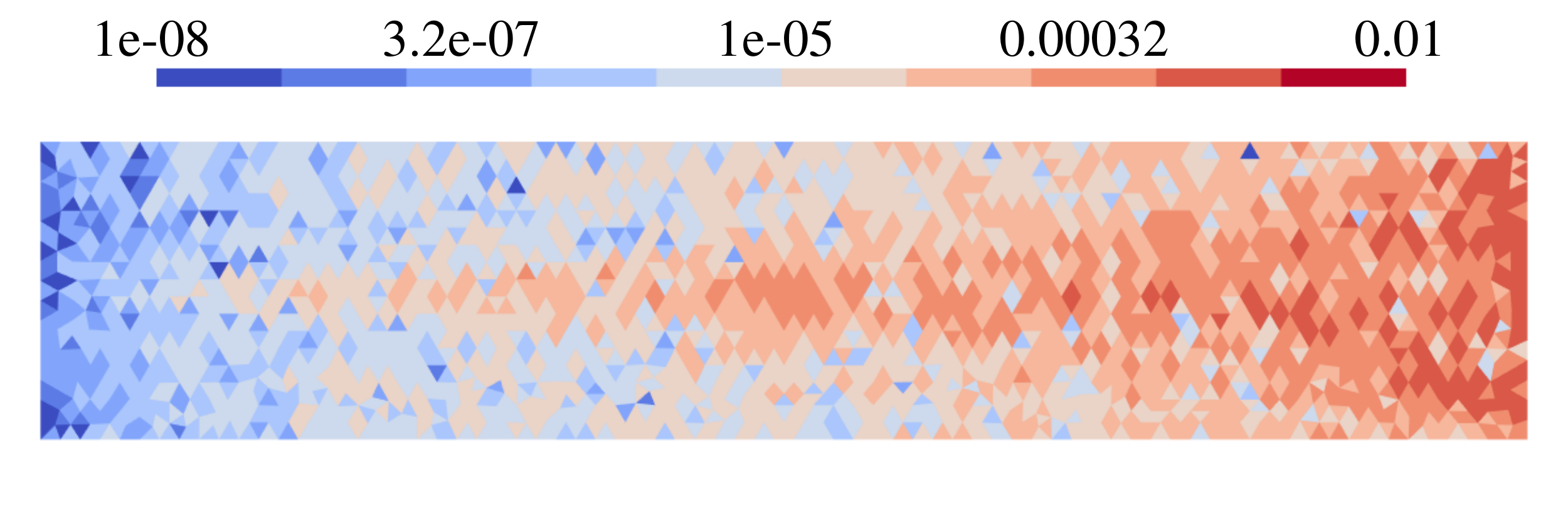}
		\caption{Rel. error estimate for $W_\text{stable}$ and ``exact'' mode.}
	\end{subfigure}

	\vspace{2ex}

	\begin{subfigure}[b]{0.48\textwidth}
		\centering
		\includegraphics[width=\textwidth]{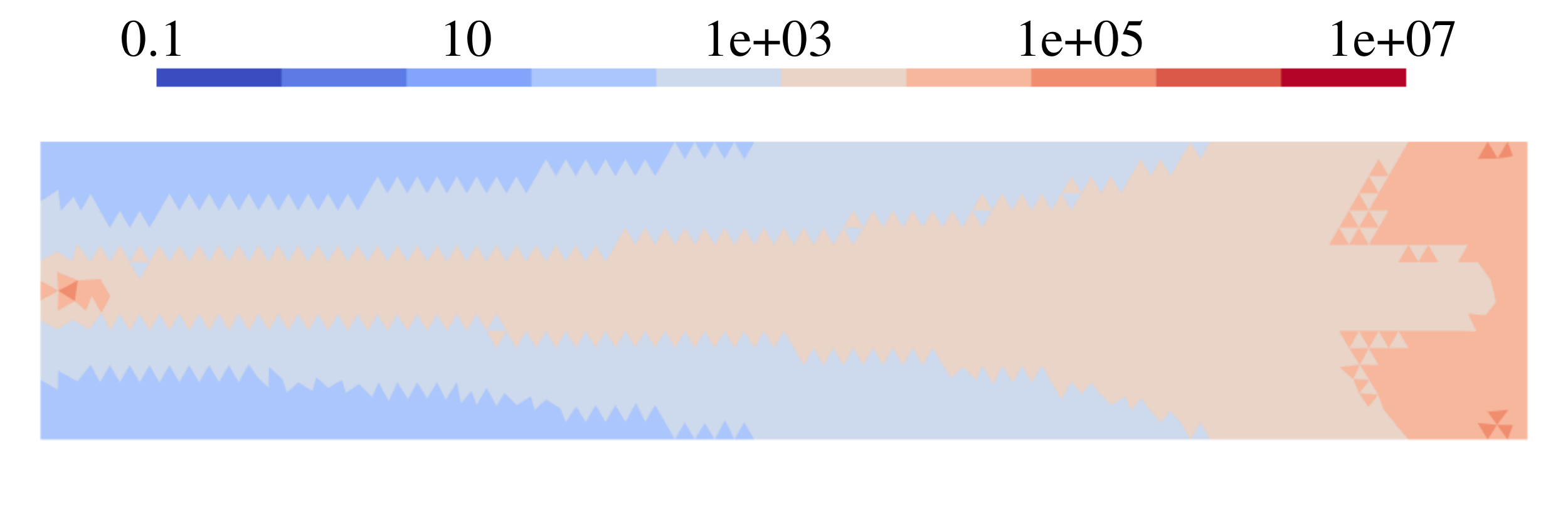}
		\caption{Rel. error bound for $W$ and ``worst'' mode.}
	\end{subfigure}
	\hfill
	\begin{subfigure}[b]{0.48\textwidth}
		\centering
		\includegraphics[width=\textwidth]{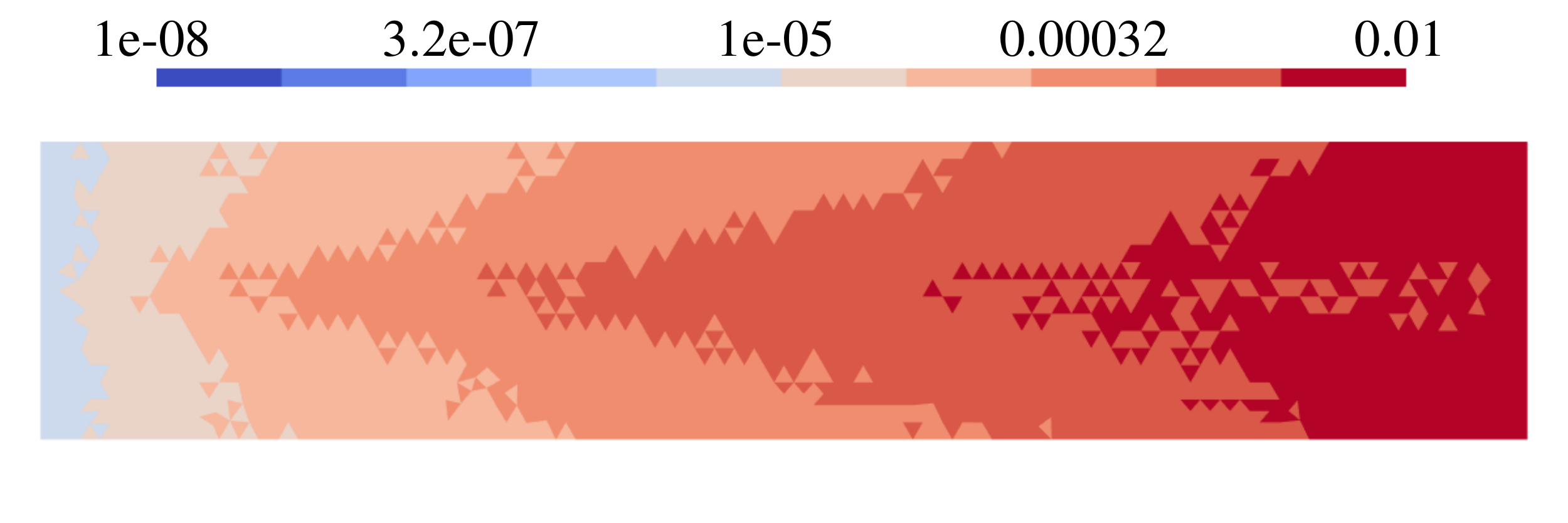}
		\caption{Rel. error bound for $W_\text{stable}$ and ``worst'' mode.}
	\end{subfigure}

	\caption{Comparison of different error estimation modes for stable and unstable Neo-Hooke strain energy density expressions.
		Figures use a logarithmic colormap.
	}
	\label{fig:neo-hooke-errors}
\end{figure}

The error estimates derived for the ``exact'' and ``worst'' modes are shown in Fig.~\ref{fig:neo-hooke-errors}.
We compute the relative rounding error estimate of the strain energy density evaluation, i.e.,
\begin{equation}
	\eta_i = \frac{|e_i|}{|b_i^\text{ref}|},
\end{equation}
where $e_i$ is the error estimate provided by the developed framework, and $b_i^\text{ref}$ is the double precision evaluation of the stable expression.
For the ``worst'' mode the estimate is a nonnegative approximate bound, $|\fl(b_i) - b_i| \le e_i + \text{h.o.t.}$, while for the ``exact'' mode the estimate has a sign, $\fl(b_i) - b_i \approx e_i$.
The left column represents the unstable expression.
The ``exact'' mode produces the relative error estimates as large as $10^{5}$, while their smallest value is not below $10^{-1}$.
These errors are large, given that we would expect an \emph{accurate} algorithm to achieve relative errors of the order of the 32-bit single precision machine epsilon $\epsmach^{(32)} \approx 1.19 \times 10^{-7}$.

The situation is much improved for the stable strain energy expression based on the expansion, shown in the right column of Fig.~\ref{fig:neo-hooke-errors}.
Indeed, the ``exact'' mode estimates relative errors in the range between $10^{-8}$ and $10^{-3}$, which is satisfactory.

\paragraph{Dependence of the error on the deformation scale}

As a second experiment, we alter the scale of the load $\alpha \in [10^{-4}, 10^{1/2}]$ in the pre-processing step.
To characterize the deformation scale, we compute a dimensionless quantity
\begin{equation}
	\Pi \coloneqq \frac{\|\bm u\|_\infty}{w} \sim C \|\nabla \bm u\|,
	\label{eq:deformation-scale}
\end{equation}
where $\|\bm u\|_\infty$ is the maximum Euclidean norm of $\bm u$ over the whole domain $\Omega$.
The last equivalence holds in the sense of the limit $\alpha \to 0$, since we are interested in the small deformation regime, and the pre-processing solution $\bm u$ has no rigid body modes.

\begin{figure}[ht]
	\centering
	\begin{tikzpicture}
		\begin{loglogaxis}[
				width=9cm, height=7cm,
				xlabel={deformation scale $\Pi \sim C\|\nabla \bm u\|$},
				ylabel={max. rel. rounding error $\eta^\text{cell}$},
				ymin=1e-9, ymax=1e14,
				grid=both,
				major grid style={dotted, gray},
				minor grid style={dotted, gray!40},
				legend cell align=left,
				legend style={
						font=\small,
						at={(1.02,0.5)},
						anchor=west,
					},
				mark size=2pt,
				only marks,
			]
			\addplot[plotblue, mark=square*]  table[x=Pi_3, y=unstable_exact] {experiments/neo_hooke_max_error_vs_load.dat};
			\addlegendentry{\texttt{std::float32\_t} unstable ``exact''}
			\addplot[plotblue, mark=square]  table[x=Pi_3, y=unstable_worst] {experiments/neo_hooke_max_error_vs_load.dat};
			\addlegendentry{\texttt{std::float32\_t} unstable ``worst''}

			\addplot[plotorange, mark=triangle*]  table[x=Pi_3, y=expansion_exact] {experiments/neo_hooke_max_error_vs_load.dat};
			\addlegendentry{\texttt{std::float32\_t} stable ``exact''}
			\addplot[plotorange, mark=triangle]  table[x=Pi_3, y=expansion_worst] {experiments/neo_hooke_max_error_vs_load.dat};
			\addlegendentry{\texttt{std::float32\_t} stable ``worst''}

			\draw[plotorange, dashed, line width=1pt] (axis cs:1e-8,1.19e-7) -- (axis cs:1e-2,1.19e-7);
			\node[plotorange, font=\scriptsize, anchor=south west] at (axis cs:1e-7,1.19e-7) {$\epsmach^{(32)}$};

			\draw[plotblue, solid, line width=1pt] (axis cs:1e-8,1e13)
			-- node[pos=0.2, below, font=\scriptsize, inner sep=6pt] {$C \Pi^{-2}$}
			(axis cs:1e-2,1e1);

		\end{loglogaxis}
	\end{tikzpicture}
	\caption{Maximum relative rounding error of the Neo-Hooke strain energy density $W$ as a function of the deformation scale $\Pi \sim C\|\nabla \bm u\|$.
		The stable, series-based expansion has for a small deformation a constant error approximation, while the unstable expression shows rounding error scaling as $\Pi^{-2}$.
		The dashed horizontal line marks the machine epsilon $\epsmach^{(32)}$ of single precision.
	}
	\label{fig:neo-hooke-err-vs-scale}
\end{figure}
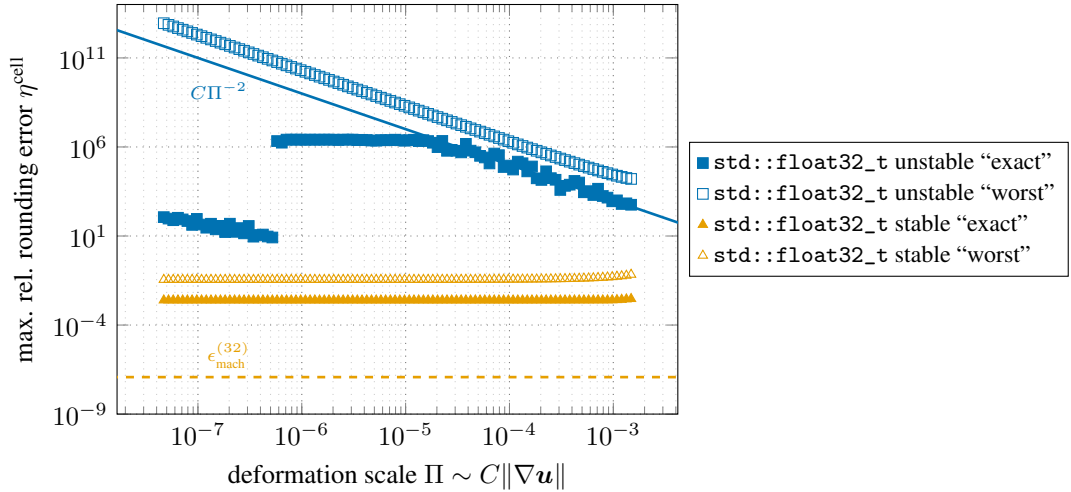

For a very small deformation, $\|\nabla \bm u\| \ll 1$, we expect the rounding errors in the computation of $\bm F = \bm I + \nabla \bm u$ to grow relative to the value of the strain energy density.
This is shown in Fig.~\ref{fig:neo-hooke-err-vs-scale}, where we plot the maximum relative error estimate $\eta^\text{cell} \coloneqq \max_i \eta_i$, with the maximum taken over all cells in the mesh.

Figure~\ref{fig:neo-hooke-err-vs-scale} confirms the earlier observation that the unstable expression has relative rounding errors dependent on the scale of the deformation, while the stable expression retains a flat behaviour.
It is measured in both modes, ``worst'' and ``exact''.
For the stable expression, both have a nearly flat curve independent of the deformation scale.
It suggests we are dealing with a problem that has a bounded condition number.

\paragraph{Conditioning}

The Neo-Hooke example is informative because it shows a problem with a bounded condition number that is implemented by a numerically unstable algorithm.
Let us consider the problem of the evaluation of the strain energy density function $W(\bm u)$ of the Neo-Hooke material as a function of the displacement field $\bm u$.
The relative condition number is usually defined as (see \citet[Eq.~(12.1)]{trefethen1997})
\begin{equation}
	\kappa_{W} \coloneqq \sup_{\delta \bm u \ne 0} \frac{|\delta W|}{|W|} \Big/ \frac{\|\delta \bm u\|}{\|\bm u\|} = \sup_{\delta \bm u \ne 0} \frac{|\delta W|}{\|\delta \bm u\|} \frac{\|\bm u\|}{|W|},
\end{equation}
where $\delta \bm u$ is an arbitrary perturbation direction and $\delta W \coloneqq DW(\bm u)[\delta \bm u] = \frac{\mathrm d}{\mathrm dt}
	W(\bm u + t \delta \bm u)\rvert_{t=0}$ is the directional derivative of $W$.
The relative condition number measures the ratio of the relative change of the perturbed output to the relative change of the perturbed input.
If the relative condition number is moderate, e.g., $\kappa_W < 100$, we say that the problem is \emph{well-conditioned}.
The choice of the norm $\|\cdot\|$ in the definition above depends on the context.
Here, we consider the condition number of the evaluation of $W$ with respect to the (finite-dimensional) coefficients of the FE function $\bm u$.

\begin{lemma}[Condition number of the strain energy density evaluation]
	\label{lem:cond-W}
	Let $\mu > 0$ and $\lambda \ge 0$, and let $\bm u \in U_h$ be a continuous, piecewise linear displacement field that satisfies the cell-wise non-rigid body modes (non-RBMs) assumption, i.e.,
	\begin{equation}
		\|\bm u\|_{\infty,\mathcal K} \le C_R \| \sym(\nabla \bm u) \|, \qquad \text{(non-RBMs assumption)}
	\end{equation}
	for a fixed $C_R > 0$, independent of the deformation scale, and each cell $\mathcal K$ in the triangulation of the domain $\Omega$, where $\|\bm u\|_{\infty,\mathcal K} \coloneqq \max_{y \in \mathcal{K}} \|\bm u(y)\|$ is the maximum Euclidean norm on the cell $\mathcal K$.
	All matrix norms are Frobenius norms.
	Then, on each fixed cell $\mathcal K$, there exists $\rho > 0$ such that the relative condition number $\kappa_W$ of the evaluation of the Neo-Hooke strain energy density $W(\bm u)$ in Eq.~\eqref{eq:energy-unstable} is bounded whenever $0 < \|\nabla \bm u\| \le \rho$, i.e.,
	\begin{equation}
		\kappa_{W} = \sup_{\delta \bm u \ne 0} \frac{|\delta W|}{|W|} \Big/ \frac{\|\delta \bm u\|_{\infty,\mathcal K}}{\|\bm u\|_{\infty,\mathcal K}} \leq C C_R h_{\mathcal K}^{-1} \eqqcolon C_W.
	\end{equation}
	The supremum is over perturbations in the local finite element space $[\mathcal P_1(\mathcal K)]^d$.
	The constants $C_W$ and $\rho$ depend on the material parameters $\mu$ and $\lambda$, dimension $d$, the non-RBMs constant $C_R$, and the cell $\mathcal K$ through its diameter $h_{\mathcal K}$ and its shape.
	In particular, the bound is independent of the deformation scale $\|\nabla \bm u\|$.
\end{lemma}
The proof is given in Appendix~\ref{app:cond-W}.

Lemma~\ref{lem:cond-W} establishes that, on a fixed mesh, the relative condition number of the strain energy density evaluation remains bounded as the deformation scale decreases, under the stated assumptions.
Consequently, a backward stable evaluation with relative input perturbations of order $\epsmach$ has relative forward error of order $C_W \epsmach$.
The lemma is stated cell-wise, and taking $C_W$ as the maximum over the cells gives a uniform condition-number bound on a fixed mesh.
The lemma asserts independence of the deformation scale rather than a small constant.
For the cantilever of Fig.~\ref{fig:cantilever-neo-hooke} the cells near the loaded end displace much more than they strain, and therefore carry a large $C_R$.
The numerical experiments in Fig.~\ref{fig:neo-hooke-err-vs-scale}, where the stable expression has an error estimate independent of the deformation scale, are consistent with this prediction.
The unstable expression scales as $\|\nabla \bm u\|^{-2}$, which could be intuitively explained: intermediate quantities such as $I_1$ and $J$ remain of order one as the deformation decreases, and their evaluation can introduce absolute rounding errors of order $\epsmach$.
These errors need not cancel when the intermediate quantities are combined to obtain the energy.
Since the exact energy decreases quadratically with the deformation scale, an absolute error of order $\epsmach$ produces relative errors of order $\epsmach \|\nabla \bm u\|^{-2}$.

\paragraph{Performance}

Performance of the rounding error estimates is essential.
In fact, if it were not for the performance considerations, one could always argue for computing a reference value in much higher precision and evaluating the exact rounding error.
As we discussed in the literature overview, the use of high-precision floating-point types is expensive.
This is the case for emulated precision, e.g., when computing the rounding error for a double precision value, but also when using a double precision value as a reference value for single precision computation.
For example, NVIDIA specifies the Rubin GPU at 4 Pflop/s for dense \texttt{float16} operations, which is a factor of 121 above the throughput of 33 Tflop/s for scalar \texttt{float64} operations, see \citet{nvidiarubin2026}.
This argument supports the use of running error bounds, since there is a larger performance margin for additional error computations.
\footnote{
	However, none of the tested FE kernels exploit the dense matrix-matrix multiplication to saturate the peak throughput, and the kernels are memory-bandwidth bound.
	Even for high-degree polynomial kernels the generated code contains mostly scalar instructions, and scalar arithmetic throughput depends on the processor, operation, and precision, so the quoted peak throughput ratio does not directly predict the benefit of reduced precision in these kernels.
}

\begin{table}[ht]
	\centering
	\renewcommand{\arraystretch}{1.5}
	\small
	\setlength{\tabcolsep}{4pt}
	\input{experiments/neo_hooke_performance_table.tex}
	\vspace{1em}
	\caption{Assembly times in milliseconds for the unstable and stable Neo-Hooke strain energy expressions, for the plain \texttt{std::float64\_t} kernel and the ``exact''/``worst'' error estimation modes.
		The ``type'' column gives the floating-point type the assembly is performed in.
		The plain kernel is always \texttt{std::float64\_t} and is shared by both rows of a group.
		The slowdown relative to the plain kernel is given in brackets.
		Values are reported as mean $\pm$ standard deviation over the 90 fastest of 100 repeated runs.
	}
	\label{tab:neo-hooke-slowdown}
\end{table}

The measured cost of the error estimation is summarized in Tab.~\ref{tab:neo-hooke-slowdown}.
Benchmarks were executed on a MacBook Pro 2024 with an Apple M4 CPU and 16 GB of memory, in an \texttt{aarch64} (ARMv8-A) Linux container running Ubuntu 26.04.1 LTS with glibc 2.43, GCC 15.2.0 and Clang 21.1.8.
The CPU provides hardware 16-bit precision arithmetic and FMA, but no hardware 128-bit format, so the 128-bit reference format of the ``exact'' mode is emulated in software by the libgcc runtime.
The ``worst'' mode is inexpensive, with a slowdown of up to $3.5\times$ over the plain \texttt{std::float64\_t} assembly.
We measure the same slowdown also for the ``exact'' mode, with the exception of three setups shown in red.
These are the runs with \texttt{std::float64\_t} and the unstable expression.
This is due to the calls to the logarithm function (only present in the unstable variant, the expansion does not have any transcendental functions) for the 128-bit reference format, which does not have an optimized implementation.
This results in a slowdown of up to $22\times$.

All slowdowns vary only mildly with the number of cells, i.e., the error estimation adds an approximately constant per-cell overhead and does not change the linear scaling of the assembly.

\subsection{Laplace operator on a near-degenerate mesh}

In this example, we focus on rounding errors introduced in the computation of geometric quantities.
Let us consider the evaluation of the weak discrete Laplace operator on the space $V_h$ of continuous Lagrange functions of degree $k=1$ on a triangulation $\mathcal T_h$ of a domain $\Omega$:
\begin{equation}
	a(u, v) \coloneqq \int_\Omega \nabla u \cdot \nabla v \, \mathrm dx,
\end{equation}
which is a bilinear form $a(\cdot, \cdot): V_h \times V_h \longrightarrow \mathbb R$.
\begin{figure}[ht]
	\centering
	\begin{tikzpicture}[scale=4, line cap=round, line join=round, >=stealth]
		\def\d{0.08}

		\coordinate (p00) at (0.00,0.00);
		\coordinate (p10) at (0.25,0.00);
		\coordinate (p20) at (0.50,0.00);
		\coordinate (p30) at (0.75,0.00);
		\coordinate (p40) at (1.00,0.00);

		\coordinate (p01) at (0.00,0.25);
		\coordinate (p11) at (0.25,0.25);
		\coordinate (p21) at (0.50,0.25);
		\coordinate (p31) at (0.75,0.25);
		\coordinate (p41) at (1.00,0.25);

		\coordinate (p02) at (0.00,0.50);
		\coordinate (p12) at (0.25,0.50);
		\coordinate (p22) at ({0.75-\d},0.50); 
		\coordinate (p32) at (0.75,0.50);
		\coordinate (p42) at (1.00,0.50);

		\coordinate (p03) at (0.00,0.75);
		\coordinate (p13) at (0.25,0.75);
		\coordinate (p23) at (0.50,0.75);
		\coordinate (p33) at (0.75,0.75);
		\coordinate (p43) at (1.00,0.75);

		\coordinate (p04) at (0.00,1.00);
		\coordinate (p14) at (0.25,1.00);
		\coordinate (p24) at (0.50,1.00);
		\coordinate (p34) at (0.75,1.00);
		\coordinate (p44) at (1.00,1.00);

		\draw (p00)--(p10)--(p20)--(p30)--(p40);
		\draw (p01)--(p11)--(p21)--(p31)--(p41);
		\draw (p02)--(p12)--(p22)--(p32)--(p42);
		\draw (p03)--(p13)--(p23)--(p33)--(p43);
		\draw (p04)--(p14)--(p24)--(p34)--(p44);

		\draw (p00)--(p01)--(p02)--(p03)--(p04);
		\draw (p10)--(p11)--(p12)--(p13)--(p14);
		\draw (p20)--(p21)--(p22)--(p23)--(p24);
		\draw (p30)--(p31)--(p32)--(p33)--(p34);
		\draw (p40)--(p41)--(p42)--(p43)--(p44);

		\draw (p00)--(p11);
		\draw (p10)--(p21);
		\draw (p20)--(p31);
		\draw (p30)--(p41);

		\draw (p01)--(p12);
		\draw (p11)--(p22);
		\draw (p21)--(p32);
		\draw (p31)--(p42);

		\draw (p02)--(p13);
		\draw (p12)--(p23);
		\draw (p22)--(p33);
		\draw (p32)--(p43);

		\draw (p03)--(p14);
		\draw (p13)--(p24);
		\draw (p23)--(p34);
		\draw (p33)--(p44);

		\fill (p21) circle (0.4pt);
		\fill (p32) circle (0.4pt);
		\fill (p22) circle (0.4pt);

		\node[font=\small, below right=1pt] at (p21) {$\bm x_0$};
		\node[font=\small, above right=1pt] at (p32) {$\bm x_1$};
		\node[font=\small, above left=1pt]  at (p22) {$\bm x_2$};

		\draw[<->]
		($(p22)+(0,-0.035)$) -- node[below=2pt] {$\Delta$} ($(p32)+(0,-0.035)$);

		\draw[->,thick] (p00) -- (0.18,0)
		node[below right] {$x$};
		\draw[->,thick] (p00) -- (0,0.18)
		node[above left] {$y$};

	\end{tikzpicture}
	\caption{Square triangular mesh with near-degenerate cells, and the distance to the degenerate mesh in unscaled coordinates denoted $\Delta$.
		Vertices of the near-degenerate cell $\mathcal K$ are $\bm x_0$, $\bm x_1$, and $\bm x_2$.
	}
	\label{fig:fan-mesh-needle}
\end{figure}
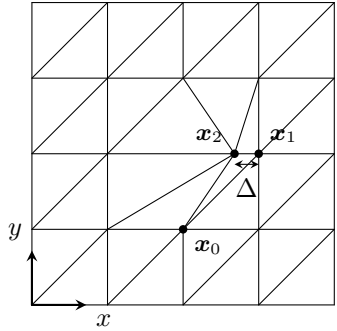

We assemble the bilinear form $a(u, v)$ on the mesh shown in Fig.~\ref{fig:fan-mesh-needle}.
The mesh is produced by a right-diagonal split of a uniform 4-by-4 square discretization of a unit square.
For the purpose of avoiding exactly representable coordinates, we scale the whole mesh by a factor of $\sqrt{2}$.
We displace the center node in the $x$ direction so that its distance to the nearest node in the unscaled mesh is $\Delta \ll 1$.
\footnote{After scaling by $\sqrt{2}$, the physical node separation is $\sqrt{2} \Delta$.}
The global matrix $\bm K \in \mathbb R^{25 \times 25}$ for a linear Lagrange discretization is
\begin{equation}
	K_{ij} \coloneqq \int_\Omega \nabla \varphi_i \cdot \nabla \varphi_j \, \mathrm dx,
\end{equation}
where $\varphi_i$ and $\varphi_j$ are the linear Lagrange basis functions.
Methods developed in this work provide the approximate error bound $|\fl(K_{ij}) - K_{ij}| \le e_{ij} + \text{h.o.t.}$ for the ``worst'' mode, or a signed error estimate $\fl(K_{ij}) - K_{ij} \approx e_{ij}$, for each element of the global matrix.
In order to visualize relative errors in the assembly process, we compute
\begin{equation}
	\text{max. rel. row-wise error }
	\eta_i^\text{row} \coloneqq \max_j \frac{|e_{ij}|}{\sqrt{K_{ii}
			K_{jj}}}.
\end{equation}
The denominator $\sqrt{K_{ii} K_{jj}}$ represents an energy-like normalization, and for the Lagrange basis $K_{ii} = \int_\Omega |\nabla \varphi_i|^2 \, \mathrm dx > 0$.
For the value of the reference $K_{ij}$ we use the matrix assembled in double precision.
The quantity $\eta_i^\text{row}$ measures the row-wise maximum energy-normalized relative error estimate.
Taking the maximum over the $i$-th row allows us to visualize the assembled value and its error bounds in a mesh plot, which is done in Fig.~\ref{fig:laplace-max}.
We can see, similarly to the previous example, that the ``exact'' error mode in Fig.~\ref{fig:laplace-max-err-exact} is much tighter than the error bound in Fig.~\ref{fig:laplace-max-err-worst}, which is based on the ``worst'' mode.

\begin{figure}[ht]
	\centering
	\begin{subfigure}[t]{0.32\textwidth}
		\centering
		\includegraphics[width=\textwidth]{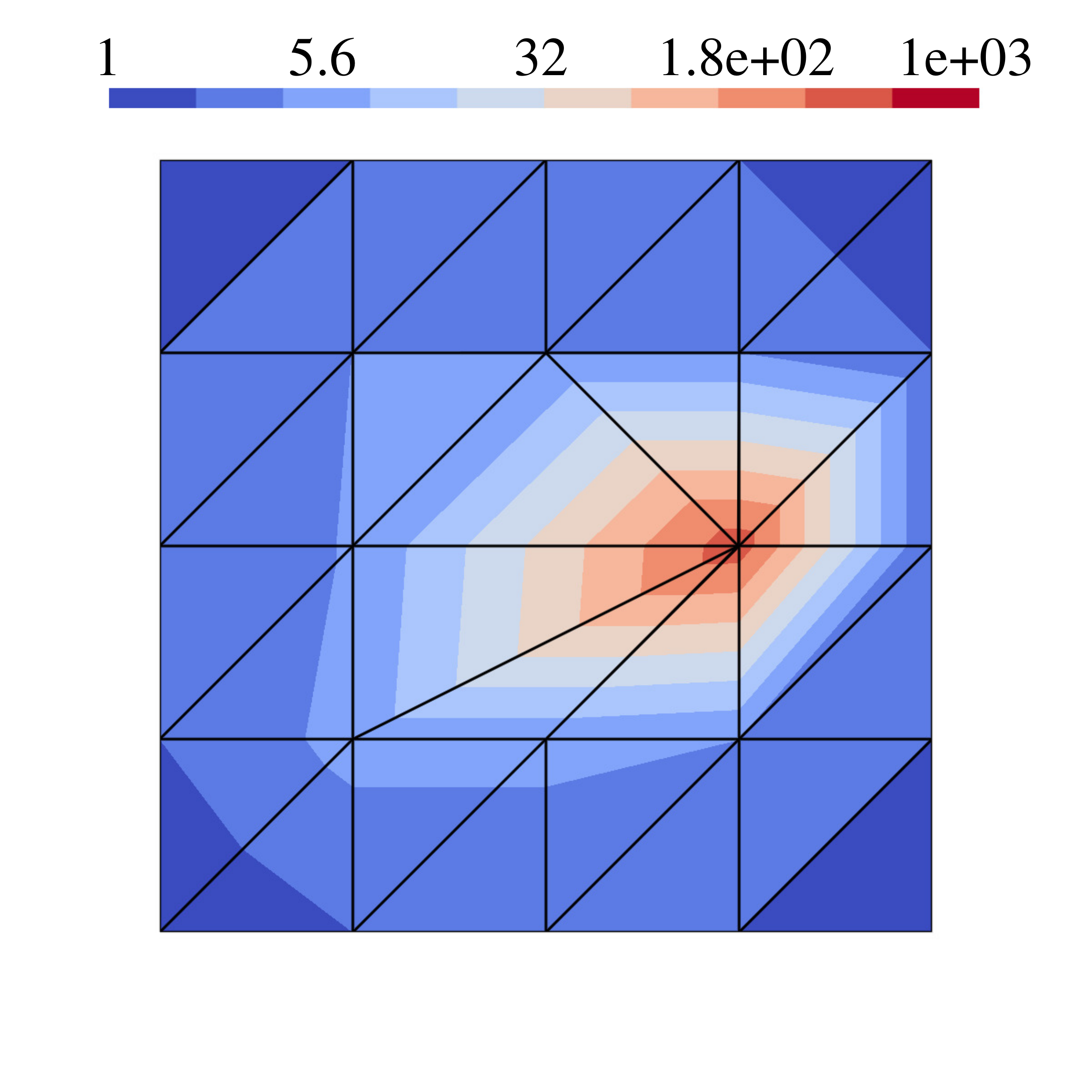}
		\caption{Max. matrix entry in the $i$-th row $\max_j |K_{ij}|$.}
		\label{fig:laplace-max-value}
	\end{subfigure}
	\hfill
	\begin{subfigure}[t]{0.32\textwidth}
		\centering
		\includegraphics[width=\textwidth]{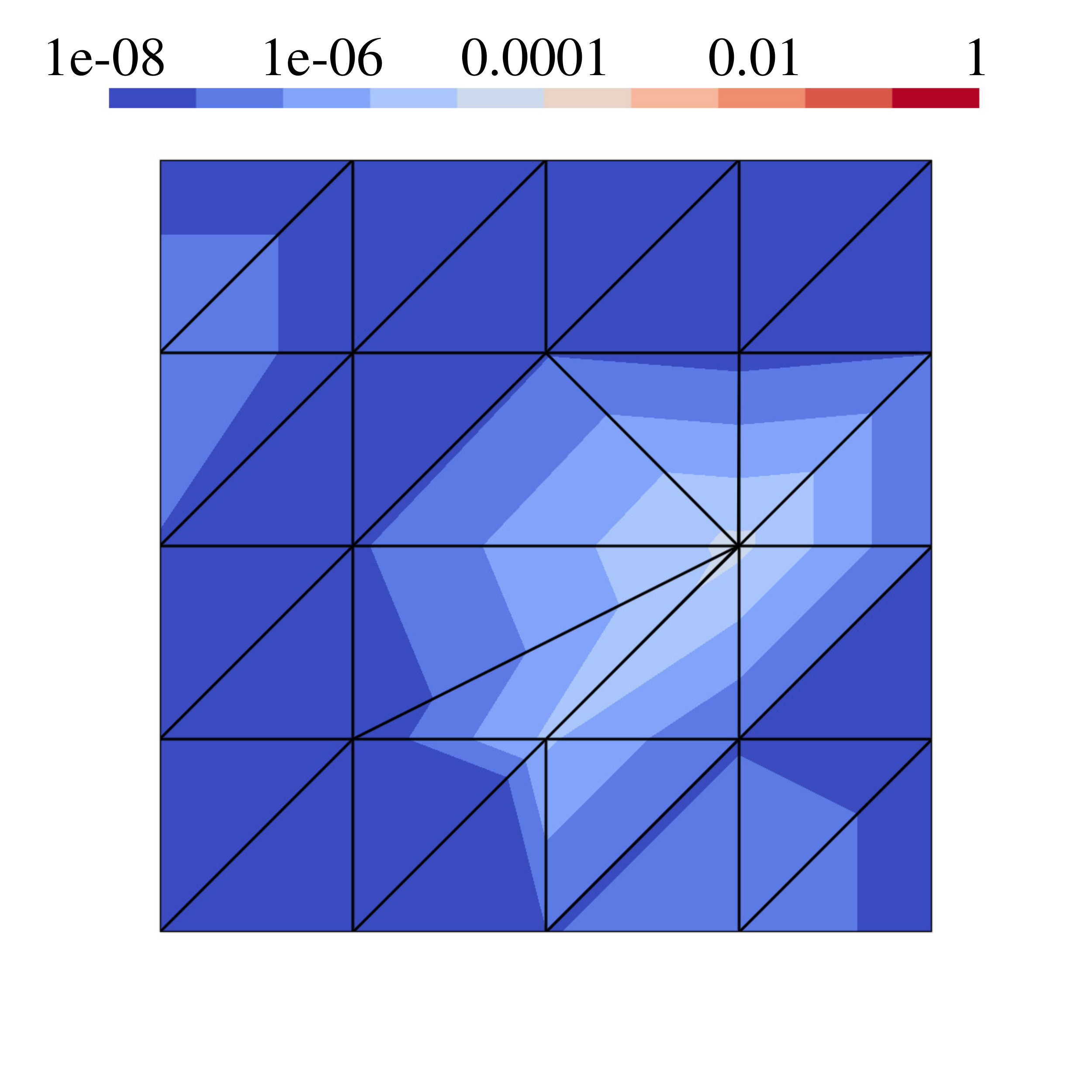}
		\caption{Max. rel. error estimate $\eta_i^\text{row}$, ``exact'' mode.}
		\label{fig:laplace-max-err-exact}
	\end{subfigure}
	\hfill
	\begin{subfigure}[t]{0.32\textwidth}
		\centering
		\includegraphics[width=\textwidth]{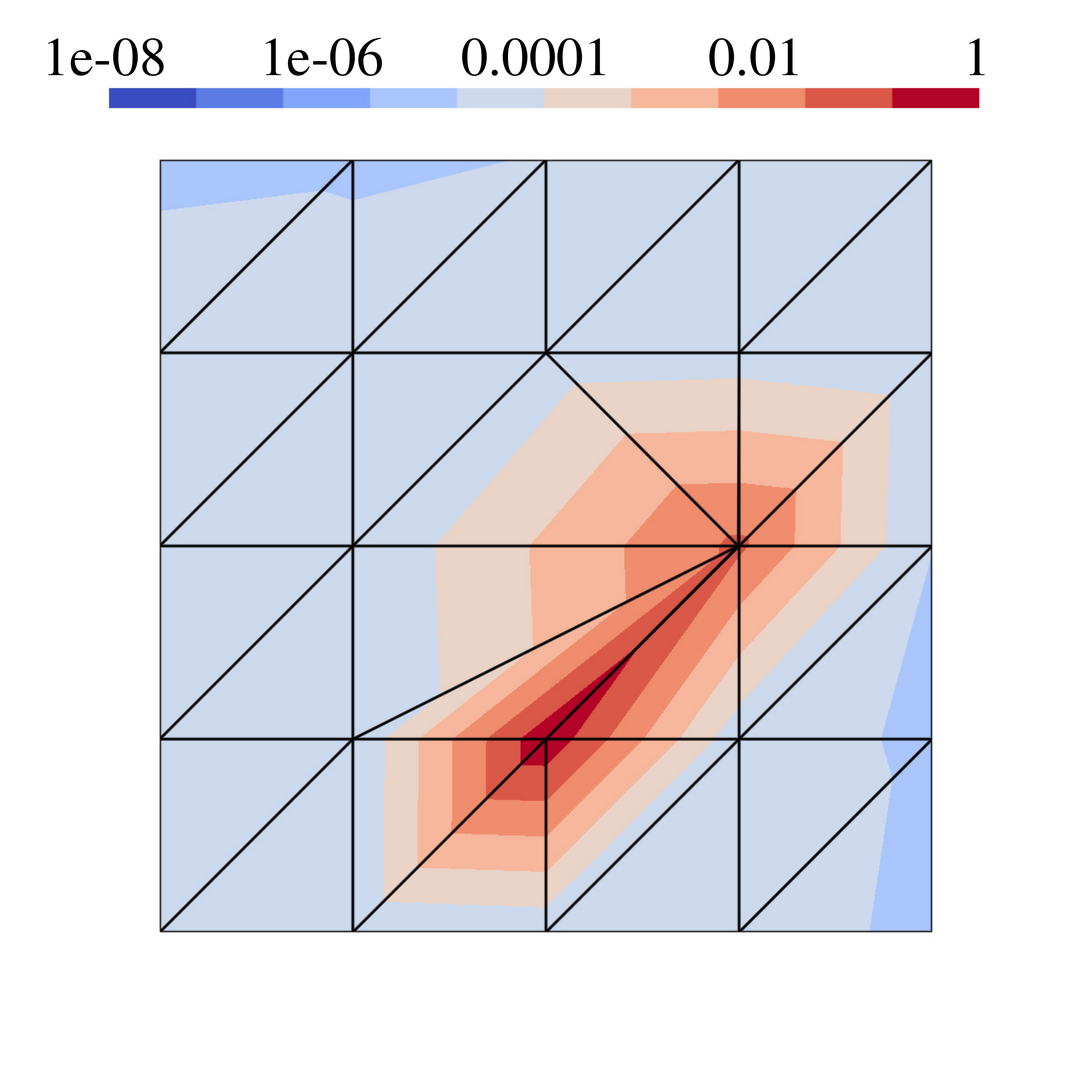}
		\caption{Max. rel. error bound $\eta_i^\text{row}$, ``worst'' mode.}
		\label{fig:laplace-max-err-worst}
	\end{subfigure}
	\caption{Row-wise maximum entries of the Laplace operator $K_{ij}$, and the maximum row-wise relative error bound/estimate, for the ``exact'' and ``worst'' modes and \texttt{std::float32\_t} single precision.
		The experiment was run for $\Delta = 10^{-3}$, so the degenerate triangle is not visible in the mesh.
		Figures use a logarithmic colormap.
	}
	\label{fig:laplace-max}
\end{figure}

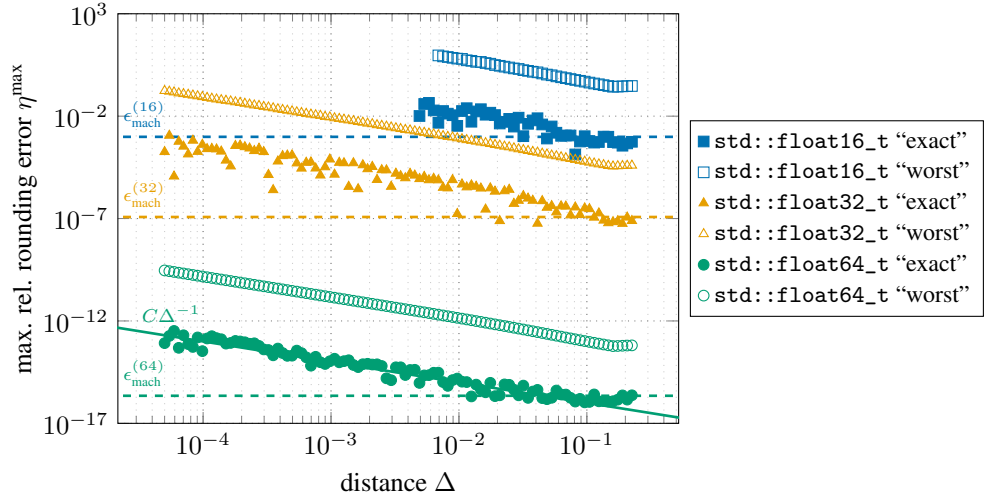
\begin{figure}[ht]
	\centering
	\begin{tikzpicture}
		\begin{loglogaxis}[
				width=9cm, height=7cm,
				xlabel={distance $\Delta$},
				ylabel={max. rel. rounding error $\eta^\text{max}$},
				ymin=1e-17, ymax=1e3,
				grid=both,
				major grid style={dotted, gray},
				minor grid style={dotted, gray!40},
				legend cell align=left,
				legend style={
						font=\small,
						at={(1.02,0.5)},
						anchor=west,
					},
				mark size=2pt,
				only marks,
			]
			\addplot[plotblue, mark=square*]  table[x=delta, y=fp16_exact] {experiments/max_laplace_err_vs_delta.dat};
			\addlegendentry{\texttt{std::float16\_t} ``exact''}
			\addplot[plotblue, mark=square]  table[x=delta, y=fp16_worst] {experiments/max_laplace_err_vs_delta.dat};
			\addlegendentry{\texttt{std::float16\_t} ``worst''}

			\addplot[plotorange, mark=triangle*]  table[x=delta, y=fp32_exact] {experiments/max_laplace_err_vs_delta.dat};
			\addlegendentry{\texttt{std::float32\_t} ``exact''}
			\addplot[plotorange, mark=triangle]  table[x=delta, y=fp32_worst] {experiments/max_laplace_err_vs_delta.dat};
			\addlegendentry{\texttt{std::float32\_t} ``worst''}

			\addplot[plotgreen, mark=*]  table[x=delta, y=fp64_exact] {experiments/max_laplace_err_vs_delta.dat};
			\addlegendentry{\texttt{std::float64\_t} ``exact''}
			\addplot[plotgreen, mark=o]  table[x=delta, y=fp64_worst] {experiments/max_laplace_err_vs_delta.dat};
			\addlegendentry{\texttt{std::float64\_t} ``worst''}
			\draw[plotblue, dashed, line width=1pt] (axis cs:1e-5,9.77e-4) -- (axis cs:1,9.77e-4);
			\node[plotblue, font=\scriptsize, anchor=south west] at (axis cs:2e-5,9.77e-4) {$\epsmach^{(16)}$};

			\draw[plotorange, dashed, line width=1pt] (axis cs:1e-5,1.19e-7) -- (axis cs:1,1.19e-7);
			\node[plotorange, font=\scriptsize, anchor=south west] at (axis cs:2e-5,1.19e-7) {$\epsmach^{(32)}$};

			\draw[plotgreen, dashed, line width=1pt] (axis cs:1e-5,2.22e-16) -- (axis cs:1,2.22e-16);
			\node[plotgreen, font=\scriptsize, anchor=south west] at (axis cs:2e-5,2.22e-16) {$\epsmach^{(64)}$};

			\draw[plotgreen, solid, line width=1pt] (axis cs:1e-5,1e-12)
			-- node[pos=0.15, above, font=\scriptsize, inner sep=5pt] {$C \Delta^{-1}$}
			(axis cs:1,1e-17);
		\end{loglogaxis}
	\end{tikzpicture}
	\caption{Maximum relative rounding error of the assembled Laplace operator as a function of the distance $\Delta$, for the \texttt{std::float16\_t}, \texttt{std::float32\_t}, and \texttt{std::float64\_t} precisions and the ``exact''/``worst'' error estimation modes.
		Dashed horizontal lines mark the machine epsilon $\epsmach$ of each precision.
	}
	\label{fig:laplace-err-vs-delta}
\end{figure}

Additionally, we measure the maximum error relative to the maximum matrix entry,
\begin{equation}
	\eta^\text{max} \coloneqq \frac{\max_{ij} |e_{ij}|}{\max_{ij} |K_{ij}|},
\end{equation}
and plot it over varying distances $\Delta$.
This is shown in Fig.~\ref{fig:laplace-err-vs-delta} for three different kernel precisions \texttt{std::float16\_t}, \texttt{std::float32\_t}, and \texttt{std::float64\_t}.
The ``exact'' error mode shows rounding errors behaving as $\Delta^{-1}$ for all three precisions, with the error being of the order of machine epsilon for distances $\Delta > 0.1$.
The ``worst'' mode is more pessimistic as expected, but more importantly it keeps the same asymptotic behaviour.
There are data-points excluded for half-precision \texttt{std::float16\_t}, for $\Delta < 5 \times 10^{-3}$, due to underflow or overflow.

\paragraph{Conditioning}

Numerical experiments from Fig.~\ref{fig:laplace-err-vs-delta} suggest that the problem of assembling the global matrix $K_{ij}(\bm x)$ from given mesh coordinates $\bm x \in \mathbb R^{25 \times 2}$ is not well-conditioned in the limit of $\Delta \to 0$, since the relative forward error grows as $\Delta^{-1}$ even in the ``exact'' mode.
This result agrees with the theoretical bounds derived in \citet{croci2024kernels}.
Theorem~4.5 in \citet{croci2024kernels} shows that the upper bound on the relative rounding error in the evaluation of the Laplace geometry tensor is proportional to the condition number of the Jacobian matrix $\bm J \coloneqq \pdv{\bm x}{\bm \xi} \in \mathbb R^{2 \times 2}$ that maps from the standard reference triangle with vertices $\bm \xi_0 = (0, 0)^\trp$, $\bm \xi_1 = (1, 0)^\trp$, and $\bm \xi_2 = (0, 1)^\trp$.
Indeed, for the needle triangle from Fig.~\ref{fig:fan-mesh-needle} we have
\begin{align}
	\bm J = \sqrt{2} \left[
		\begin{array}{c|c}
			1/4 & 1/4 - \Delta \\ \hline
			1/4 & 1/4
		\end{array}
		\right], \quad
	\kappa_2(\bm J) \sim C \Delta^{-1}
\end{align}
which matches the measured relative error in Fig.~\ref{fig:laplace-err-vs-delta}.

\section{Conclusion}
\label{sec:conclusion}

This study investigates the rounding errors in the assembly of finite element kernels.
We rely on a posteriori techniques that compute the error estimate alongside the value, such as running error analysis.
We have implemented a custom C++ data type \texttt{running\_error\_t} which tracks the rounding error in two modes: ``worst'' for the running error bounds and ``exact'' for a tighter signed error estimate based on a higher-precision floating-point format and error-free transformations.
The \texttt{running\_error\_t} is passed to the C++ kernel that is generated by the FFCx library, and linked into the FEniCS assemblers.

We tested the framework on two examples.
The na\"ive implementation of the strain energy density function of a Neo-Hooke hyperelastic material is expected to be numerically unstable in the regime of small deformations.
However, the underlying problem has a bounded condition number.
Both ``worst'' and ``exact'' error modes capture that.
In addition, a solution based on a series expansion provides improved accuracy, and both error modes suggest it.
The performance of both modes is promising.
In most cases they show a slowdown of up to $3.5\times$ compared to the plain assembly.
This makes the proposed methods usable in practice with small overhead.

The second example is the assembly of a Laplace operator on a mesh that includes near-degenerate, needle-like triangles.
We show that our error bounds estimate that rounding errors are growing with decreasing distance to the degenerate mesh.
Both the ``worst'' and the ``exact'' error modes match the asymptotic scaling predicted by the conditioning arguments.
The ``worst'' mode captures the worst case and ignores error cancellations, while the ``exact'' mode successfully captures smaller rounding error estimates arising from stochastic behaviour of roundoffs.

There are several limitations to the present work.
The code generated by FFCx does not give fine control over the floating-point type.
It distinguishes between the working type into which the scalars, vectors, or matrices are assembled, and the type for the geometry evaluation.
For example, it is not possible to decide to perform evaluation of the geometric matrix using matrix multiply-accumulate with low-precision input, while accumulating into a high-precision result.
In addition, one cannot choose to evaluate some parts of the bilinear form using high-precision computation (e.g., a non-linear material law), and others using low-precision numbers (e.g., a tabulation of basis functions).

Future work could focus on more robust handling of floating-point range effects in low-precision computations, such as overflow in the evaluation of geometry expressions, through suitable scaling and squeezing techniques.
Another direction is to choose the precision of each operation dynamically based on the available running error estimate.

\section*{Acknowledgement}

Matteo Croci has received support from the grants PID2023-146668OA-I00, RYC2022-036312-I, and CEX2021-001142-S, funded by MICIU/AEI/10.13039/501100011033 and cofunded by ESF+ and the EU.
Matteo Croci is also supported by the Basque Government through the BERC 2026-2029 program.
Paul T.~Kühner and Andreas Zilian have received partial funding from the European Union and the European Defence Agency through project IMPACT I.

\appendix
\section{Proof of Lemma~\ref{lem:cond-W}}
\label{app:cond-W}

\begin{proof}
	Perturbations of the trace $I_1$ and determinant $J$ follow
	\begin{equation}
		\begin{aligned}
			\delta I_1 & = \tr(\nabla \delta \bm u^\trp \bm F) + \tr(\bm F^\trp \nabla \delta \bm u) = 2 \tr(\bm F^\trp \nabla \delta \bm u), \\
			\delta J   & = J \tr(\bm F^{-1} \nabla \delta \bm u),
		\end{aligned}
	\end{equation}
	and we can compute the perturbation in the strain energy density in Eq.~\eqref{eq:energy-unstable} as
	\begin{equation}
		\begin{aligned}
			\delta W & = \frac{\mu}{2} \Big( \delta I_1 - \frac{2}{J} \delta J \Big) + \lambda (J - 1) \delta J                                                            \\
			         & = \mu \Big( \tr(\bm F^\trp \nabla \delta \bm u) - \tr(\bm F^{-1} \nabla \delta \bm u) \Big) + \lambda (J - 1) J \tr(\bm F^{-1} \nabla \delta \bm u) \\
			         & = \mu \Big( (\bm F - \bm F^{-\trp}) \colon \nabla \delta \bm u \Big) + \lambda (J - 1) J \bm F^{-\trp} \colon \nabla \delta \bm u
		\end{aligned}
	\end{equation}
	with an upper bound
	\begin{equation}
		|\delta W| \le \mu \| \bm F - \bm F^{-\trp} \| \| \nabla \delta \bm u \| + \lambda |J - 1| |J| \| \bm F^{-\trp} \| \| \nabla \delta \bm u \|.
	\end{equation}
	On a fixed cell $\mathcal K$, the finite element inverse inequality gives $\|\nabla \bm v\| \le C_I h_{\mathcal K}^{-1} \|\bm v\|_{\infty,\mathcal K}$ for every $\bm v \in [\mathcal P_1(\mathcal K)]^d$, see \citet{brenner2008fem}.
	Setting $\bm H \coloneqq \nabla \bm u$, $\bm S \coloneqq \sym(\bm H)$, and $M \coloneqq C_I C_R h_{\mathcal K}^{-1}$, the non-RBMs assumption therefore gives
	\begin{equation}
		\|\bm H\| \le C_I h_{\mathcal K}^{-1} \|\bm u\|_{\infty,\mathcal K} \le M \|\bm S\|.
	\end{equation}
	We can continue with bounds
	\begin{equation}
		\|\bm F - \bm F^{-\trp} \| = \| \bm I + \nabla \bm u  - (\bm I + \nabla \bm u )^{-\trp}\| = \| \nabla \bm u + \nabla \bm u^\trp + \mathcal O(\|\nabla \bm u\|^2) \| \le C \|\sym(\nabla \bm u)\| + \mathcal O(\|\nabla \bm u\|^2),
	\end{equation}
	and
	\begin{equation}
		\begin{aligned}
			|J - 1|           & = |1 + \tr(\nabla \bm u) + \mathcal O(\|\nabla \bm u\|^2) - 1| \le C \| \sym(\nabla \bm u) \| + \mathcal O(\|\nabla \bm u\|^2), \\
			|J|               & \le 1 + \varepsilon,                                                                                                            \\
			\|\bm F^{-\trp}\| & \le \sqrt{2} + \varepsilon,
		\end{aligned}
	\end{equation}
	for sufficiently small $\|\nabla \bm u\|$ and $\varepsilon \ll 1$, since we are considering the limiting case $\|\nabla \bm u\| \to 0$.
	Since $\|\bm H\|^2 \le M\|\bm H\|\|\bm S\|$, the remainder terms can be absorbed into $C\|\bm S\|$ for sufficiently small $\|\bm H\|$.
	For the condition number we then have
	\begin{equation}
		\kappa_W \le C \|\bm S\| \frac{\|\bm u\|_{\infty,\mathcal K}}{|W|} \sup_{\delta \bm u \ne 0} \frac{\|\nabla \delta \bm u\|}{\|\delta \bm u\|_{\infty,\mathcal K}}.
		\label{eq:neo-hooke-cond-bound}
	\end{equation}
	The quadratic energy expansion gives
	\begin{equation}
		W = \mu\|\bm S\|^2 + \frac{\lambda}{2}(\tr\bm S)^2 + \mathcal O(\|\bm H\|^3).
	\end{equation}
	Since $\|\bm H\|^3 \le M^2\|\bm H\|\|\bm S\|^2$, the remainder can be absorbed into the quadratic term for sufficiently small $\|\bm H\|$, giving $W \ge c\|\bm S\|^2$ for some $c > 0$.
	Substituting this bound into Eq.~\eqref{eq:neo-hooke-cond-bound} and applying the inverse inequality and the non-RBMs assumption we finally have
	\begin{equation}
		\kappa_W \le C C_I h_{\mathcal K}^{-1} \frac{\|\bm u\|_{\infty,\mathcal K}}{\|\sym(\nabla \bm u)\|} \le C C_I C_R h_{\mathcal K}^{-1} \eqqcolon C_W.
	\end{equation}
\end{proof}

\printbibliography

\end{document}

%% file: experiments/neo_hooke_performance_table.tex
\begin{tabular}{>{\columncolor{tabgray}}l|>{\columncolor{tabgray}}l||l|l|l||l|l|l}
    \rowcolor{tabgray}
     &  & \multicolumn{3}{c||}{unstable} & \multicolumn{3}{c}{stable} \\
    \rowcolor{tabgray}
    \# cells & type & plain & ``worst'' & ``exact'' & plain & ``worst'' & ``exact'' \\ \hline\hline
    \multirow{2}{*}{2708}
     & \texttt{std::float64\_t} & \multirow{2}{*}{\makecell[l]{$0.083$ \\ $\pm\, 0.006$}} & \makecell[l]{$0.149 \pm 0.004$ \\ ($\bm{1.80\times}$)} & \cellcolor{tabred}\makecell[l]{$1.29 \pm 0.05$ \\ ($\bm{15.57\times}$)} & \multirow{2}{*}{\makecell[l]{$0.095$ \\ $\pm\, 0.002$}} & \makecell[l]{$0.253 \pm 0.011$ \\ ($\bm{2.67\times}$)} & \makecell[l]{$0.183 \pm 0.007$ \\ ($\bm{1.93\times}$)} \\ \cline{2-2}\cline{4-5}\cline{7-8}
     & \texttt{std::float32\_t} &  & \makecell[l]{$0.167 \pm 0.005$ \\ ($\bm{2.01\times}$)} & \makecell[l]{$0.144 \pm 0.002$ \\ ($\bm{1.74\times}$)} &  & \makecell[l]{$0.328 \pm 0.043$ \\ ($\bm{3.46\times}$)} & \makecell[l]{$0.198 \pm 0.018$ \\ ($\bm{2.09\times}$)} \\ \hline
    \multirow{2}{*}{29030}
     & \texttt{std::float64\_t} & \multirow{2}{*}{\makecell[l]{$0.638$ \\ $\pm\, 0.012$}} & \makecell[l]{$1.66 \pm 0.10$ \\ ($\bm{2.60\times}$)} & \cellcolor{tabred}\makecell[l]{$13.62 \pm 0.58$ \\ ($\bm{21.34\times}$)} & \multirow{2}{*}{\makecell[l]{$0.882$ \\ $\pm\, 0.027$}} & \makecell[l]{$2.74 \pm 0.32$ \\ ($\bm{3.10\times}$)} & \makecell[l]{$1.77 \pm 0.04$ \\ ($\bm{2.01\times}$)} \\ \cline{2-2}\cline{4-5}\cline{7-8}
     & \texttt{std::float32\_t} &  & \makecell[l]{$1.78 \pm 0.08$ \\ ($\bm{2.78\times}$)} & \makecell[l]{$1.58 \pm 0.09$ \\ ($\bm{2.48\times}$)} &  & \makecell[l]{$2.61 \pm 0.02$ \\ ($\bm{2.95\times}$)} & \makecell[l]{$1.86 \pm 0.03$ \\ ($\bm{2.11\times}$)} \\ \hline
    \multirow{2}{*}{116054}
     & \texttt{std::float64\_t} & \multirow{2}{*}{\makecell[l]{$2.43$ \\ $\pm\, 0.02$}} & \makecell[l]{$6.09 \pm 0.06$ \\ ($\bm{2.50\times}$)} & \cellcolor{tabred}\makecell[l]{$52.39 \pm 0.24$ \\ ($\bm{21.54\times}$)} & \multirow{2}{*}{\makecell[l]{$3.17$ \\ $\pm\, 0.02$}} & \makecell[l]{$10.07 \pm 0.06$ \\ ($\bm{3.18\times}$)} & \makecell[l]{$6.98 \pm 0.04$ \\ ($\bm{2.20\times}$)} \\ \cline{2-2}\cline{4-5}\cline{7-8}
     & \texttt{std::float32\_t} &  & \makecell[l]{$6.61 \pm 0.06$ \\ ($\bm{2.72\times}$)} & \makecell[l]{$5.82 \pm 0.05$ \\ ($\bm{2.39\times}$)} &  & \makecell[l]{$10.39 \pm 0.05$ \\ ($\bm{3.28\times}$)} & \makecell[l]{$7.39 \pm 0.05$ \\ ($\bm{2.33\times}$)} \\
\end{tabular}

%% file: references.bib
@misc{croci2024kernels,
      author        = {Matteo Croci and Garth N. Wells},
      title         = {Mixed-precision finite element kernels and assembly: rounding error analysis and hardware acceleration},
      year          = {2024},
      eprint        = {2410.12614},
      archiveprefix = {arXiv},
      primaryclass  = {math.NA},
      doi           = {10.48550/arXiv.2410.12614},
      url           = {https://arxiv.org/abs/2410.12614}
}

@article{shakeri2024stabilenumerics,
      author        = {Rezgar Shakeri and Leila Ghaffari and Jeremy L. Thompson and Jed Brown},
      title         = {Stable numerics for finite-strain elasticity},
      journal       = {International Journal for Numerical Methods in Engineering},
      year          = {2024},
      volume        = {125},
      number        = {21},
      pages         = {e7563},
      doi           = {10.1002/nme.7563},
      url           = {https://doi.org/10.1002/nme.7563}
}

@article{laughton2022barycentric,
      author        = {Edward Laughton and Vidhi Zala and Akil Narayan and Robert M. Kirby and David Moxey},
      title         = {Fast barycentric-based evaluation over spectral/$hp$ elements},
      journal       = {Journal of Scientific Computing},
      year          = {2022},
      volume        = {90},
      number        = {2},
      pages         = {78},
      doi           = {10.1007/s10915-021-01750-2},
      url           = {https://doi.org/10.1007/s10915-021-01750-2}
}

@article{brubeck2025fiat,
      author        = {Pablo D. Brubeck and Robert C. Kirby and Fabian Laakmann and Lawrence Mitchell},
      title         = {{FIAT}: improving performance and accuracy for high-order finite elements},
      journal       = {ACM Transactions on Mathematical Software},
      year          = {2025},
      volume        = {51},
      number        = {3},
      articleno     = {21},
      pages         = {1--17},
      numpages      = {17},
      publisher     = {Association for Computing Machinery},
      doi           = {10.1145/3748816},
      url           = {https://doi.org/10.1145/3748816}
}

@article{higham2004barycentric,
      author        = {Nicholas J. Higham},
      title         = {The numerical stability of barycentric {Lagrange} interpolation},
      journal       = {IMA Journal of Numerical Analysis},
      year          = {2004},
      volume        = {24},
      number        = {4},
      pages         = {547--556},
      doi           = {10.1093/imanum/24.4.547},
      url           = {https://doi.org/10.1093/imanum/24.4.547}
}

@article{berrut2004interpolation,
      author        = {Jean-Paul Berrut and Lloyd N. Trefethen},
      title         = {Barycentric {Lagrange} interpolation},
      journal       = {SIAM Review},
      year          = {2004},
      volume        = {46},
      number        = {3},
      pages         = {501--517},
      doi           = {10.1137/S0036144502417715},
      url           = {https://doi.org/10.1137/S0036144502417715}
}

@techreport{melosh1969manipulationerrors,
      author        = {R. J. Melosh and E. L. Palacol},
      title         = {Manipulation errors in finite element analysis of structures},
      institution   = {National Aeronautics and Space Administration},
      number        = {NASA CR-1385},
      address       = {Washington, DC},
      month         = aug,
      year          = {1969},
      url           = {https://ntrs.nasa.gov/citations/19690024982}
}

@article{melosh1973inheritederror,
      author        = {R. J. Melosh},
      title         = {Inherited error in finite element analyses of structures},
      journal       = {Computers \& Structures},
      year          = {1973},
      volume        = {3},
      number        = {5},
      pages         = {1205--1217}
}

@article{utku1984solutionerrors,
      author        = {S. Utku and R. J. Melosh},
      title         = {Solution errors in finite element analysis},
      journal       = {Computers \& Structures},
      year          = {1984},
      volume        = {18},
      number        = {3},
      pages         = {379--393}
}

@article{fried1986roundoff,
      author        = {Isaac Fried},
      title         = {Round-off errors in the stiffness equation},
      journal       = {Computer Methods in Applied Mechanics and Engineering},
      year          = {1986},
      volume        = {57},
      number        = {2},
      doi           = {10.1016/0045-7825(86)90017-4},
      url           = {https://doi.org/10.1016/0045-7825(86)90017-4}
}

@article{babuska1981pversion,
      author        = {Ivo Babu{\v{s}}ka and Barna A. Szab{\'o} and I. Norman Katz},
      title         = {The $p$-version of the finite element method},
      journal       = {SIAM Journal on Numerical Analysis},
      year          = {1981},
      volume        = {18},
      number        = {3},
      pages         = {515--544},
      doi           = {10.1137/0718033},
      url           = {https://doi.org/10.1137/0718033}
}

@article{liu2021balancing,
      author        = {Jie Liu and Matthias M{\"o}ller and Henk M. Schuttelaars},
      title         = {Balancing truncation and round-off errors in {FEM}: one-dimensional analysis},
      journal       = {Journal of Computational and Applied Mathematics},
      year          = {2021},
      volume        = {386},
      pages         = {113219},
      doi           = {10.1016/j.cam.2020.113219},
      url           = {https://doi.org/10.1016/j.cam.2020.113219}
}

@article{alvarez2012roundoff,
      author        = {J. Alvarez-Aramberri and D. Pardo and Maciej Paszy{\'n}ski and Nathan Collier and Lisandro Dalcin and Victor M. Calo},
      title         = {On round-off error for adaptive finite element methods},
      journal       = {Procedia Computer Science},
      year          = {2012},
      volume        = {9},
      pages         = {1474--1483},
      doi           = {10.1016/j.procs.2012.04.162},
      url           = {https://doi.org/10.1016/j.procs.2012.04.162}
}

@article{babuska2018roundoff,
      author        = {Ivo Babu{\v{s}}ka and Gustaf S{\"o}derlind},
      title         = {On roundoff error growth in elliptic problems},
      journal       = {ACM Transactions on Mathematical Software},
      year          = {2018},
      volume        = {44},
      number        = {3},
      articleno     = {33},
      pages         = {1--22},
      numpages      = {22},
      publisher     = {Association for Computing Machinery},
      doi           = {10.1145/3134444},
      url           = {https://doi.org/10.1145/3134444}
}

@inproceedings{jouppi2017tpu,
      author        = {Norman P. Jouppi and others},
      title         = {In-datacenter performance analysis of a tensor processing unit},
      booktitle     = {Proceedings of the 44th Annual International Symposium on Computer Architecture},
      year          = {2017},
      pages         = {1--12},
      publisher     = {Association for Computing Machinery},
      doi           = {10.1145/3079856.3080246},
      url           = {https://doi.org/10.1145/3079856.3080246}
}

@misc{micikevicius2022fp8,
      author        = {Paulius Micikevicius and Dusan Stosic and Neil Burgess and Marius Cornea and Pradeep Dubey and Richard Grisenthwaite and Sangwon Ha and Alexander Heinecke and Patrick Judd and John Kamalu and Naveen Mellempudi and Stuart Oberman and Mohammad Shoeybi and Michael Siu and Hao Wu},
      title         = {{FP8} formats for deep learning},
      year          = {2022},
      eprint        = {2209.05433},
      archiveprefix = {arXiv},
      primaryclass  = {cs.LG},
      doi           = {10.48550/arXiv.2209.05433},
      url           = {https://arxiv.org/abs/2209.05433}
}

@article{abdelfattah2021survey,
      author        = {Ahmad Abdelfattah and Hartwig Anzt and Erik G. Boman and Erin Carson and Terry Cojean and Jack Dongarra and Alyson Fox and Mark Gates and Nicholas J. Higham and Xiaoye S. Li and Jennifer Loe and Piotr Luszczek and Srikara Pranesh and Siva Rajamanickam and Tobias Ribizel and Barry F. Smith and Kasia Swirydowicz and Stephen Thomas and Stanimire Tomov and Yaohung M. Tsai and Ulrike Meier Yang},
      title         = {A survey of numerical linear algebra methods utilizing mixed-precision arithmetic},
      journal       = {The International Journal of High Performance Computing Applications},
      year          = {2021},
      volume        = {35},
      number        = {4},
      pages         = {344--369},
      doi           = {10.1177/10943420211003313},
      url           = {https://doi.org/10.1177/10943420211003313}
}

@article{higham2022mixedprecision,
      author        = {Nicholas J. Higham and Theo Mary},
      title         = {Mixed precision algorithms in numerical linear algebra},
      journal       = {Acta Numerica},
      year          = {2022},
      volume        = {31},
      pages         = {347--414},
      doi           = {10.1017/S0962492922000022},
      url           = {https://doi.org/10.1017/S0962492922000022}
}

@article{williams2009roofline,
      author        = {Samuel Williams and Andrew Waterman and David Patterson},
      title         = {Roofline: an insightful visual performance model for multicore architectures},
      journal       = {Communications of the ACM},
      year          = {2009},
      volume        = {52},
      number        = {4},
      pages         = {65--76},
      doi           = {10.1145/1498765.1498785},
      url           = {https://doi.org/10.1145/1498765.1498785}
}

@book{wriggers2006contact,
      author        = {Peter Wriggers},
      title         = {Computational Contact Mechanics},
      edition       = {2},
      publisher     = {Springer},
      address       = {Berlin, Heidelberg},
      year          = {2006},
      doi           = {10.1007/978-3-540-32609-0},
      url           = {https://doi.org/10.1007/978-3-540-32609-0}
}

@article{simo1988multisurface,
      author        = {Juan C. Simo and John G. Kennedy and Sanjay Govindjee},
      title         = {Non-smooth multisurface plasticity and viscoplasticity: loading/unloading conditions and numerical algorithms},
      journal       = {International Journal for Numerical Methods in Engineering},
      year          = {1988},
      volume        = {26},
      number        = {10},
      pages         = {2161--2185},
      doi           = {10.1002/nme.1620261003},
      url           = {https://doi.org/10.1002/nme.1620261003}
}

@book{simo1998inelasticity,
      author        = {Juan C. Simo and Thomas J. R. Hughes},
      title         = {Computational Inelasticity},
      series        = {Interdisciplinary Applied Mathematics},
      volume        = {7},
      publisher     = {Springer},
      address       = {New York},
      year          = {1998},
      doi           = {10.1007/b98904},
      url           = {https://doi.org/10.1007/b98904}
}

@inproceedings{shewchuk2002elementquality,
      author        = {Jonathan Richard Shewchuk},
      title         = {What is a good linear finite element? Interpolation, conditioning, anisotropy, and quality measures},
      booktitle     = {Proceedings of the 11th International Meshing Roundtable},
      year          = {2002},
      pages         = {115--126},
      publisher     = {Sandia National Laboratories},
      url           = {https://people.eecs.berkeley.edu/~jrs/papers/elemj.pdf}
}

@article{kirby2006compiler,
      author        = {Robert C. Kirby and Anders Logg},
      title         = {A compiler for variational forms},
      journal       = {ACM Transactions on Mathematical Software},
      year          = {2006},
      volume        = {32},
      number        = {3},
      pages         = {417--444},
      doi           = {10.1145/1163641.1163644},
      url           = {https://doi.org/10.1145/1163641.1163644}
}

@article{alnaes2014ufl,
      author        = {Martin S. Aln{\ae}s and Anders Logg and Kristian B. {\O}lgaard and Marie E. Rognes and Garth N. Wells},
      title         = {{Unified Form Language}: a domain-specific language for weak formulations of partial differential equations},
      journal       = {ACM Transactions on Mathematical Software},
      year          = {2014},
      volume        = {40},
      number        = {2},
      articleno     = {9},
      pages         = {1--37},
      doi           = {10.1145/2566630},
      url           = {https://doi.org/10.1145/2566630}
}

@book{higham2002accuracy,
      author        = {Nicholas J. Higham},
      title         = {Accuracy and Stability of Numerical Algorithms},
      edition       = {2},
      publisher     = {Society for Industrial and Applied Mathematics},
      address       = {Philadelphia},
      year          = {2002},
      doi           = {10.1137/1.9780898718027},
      url           = {https://doi.org/10.1137/1.9780898718027}
}

@article{zahradnicky2010running,
      author        = {Tom{\'a}{\v{s}} Zahradnick{\'y} and Robert L{\'o}rencz},
      title         = {{FPU}-supported running error analysis},
      journal       = {Acta Polytechnica},
      year          = {2010},
      volume        = {50},
      number        = {2},
      pages         = {30--36},
      doi           = {10.14311/1167},
      url           = {https://doi.org/10.14311/1167}
}

@article{wilkinson1986erroranalysis,
      author        = {James H. Wilkinson},
      title         = {Error analysis revisited},
      journal       = {IMA Bulletin},
      year          = {1986},
      volume        = {22},
      number        = {11/12},
      pages         = {192--200},
      url           = {https://ima.org.uk/24908/error-analysis-revisited-1986/}
}

@incollection{braconnier2002cena,
      author        = {Thierry Braconnier and Philippe Langlois},
      title         = {From rounding error estimation to automatic correction with automatic differentiation},
      booktitle     = {Automatic Differentiation of Algorithms: From Simulation to Optimization},
      editor        = {George Corliss and Christ{\`e}le Faure and Andreas Griewank and Laurent Hasco{\"e}t and Uwe Naumann},
      pages         = {351--357},
      publisher     = {Springer},
      address       = {New York, NY},
      year          = {2002},
      doi           = {10.1007/978-1-4613-0075-5_42}
}

@incollection{hovland2021forwardcena,
      author        = {Paul D. Hovland and Jan H{\"u}ckelheim},
      title         = {Error estimation and correction using the forward {CENA} method},
      booktitle     = {Computational Science -- {ICCS} 2021},
      series        = {Lecture Notes in Computer Science},
      volume        = {12742},
      pages         = {765--778},
      publisher     = {Springer},
      address       = {Cham},
      year          = {2021},
      doi           = {10.1007/978-3-030-77961-0_61}
}

@article{demeure2022shaman,
      author        = {Nestor Demeure and C{\'e}dric Chevalier and Christophe Denis and Pierre Dossantos-Uzarralde},
      title         = {Algorithm 1029: encapsulated error, a direct approach to evaluate floating-point accuracy},
      journal       = {ACM Transactions on Mathematical Software},
      volume        = {48},
      number        = {4},
      articleno     = {47},
      pages         = {47:1--47:16},
      year          = {2022},
      doi           = {10.1145/3549205},
      note          = {Presents the Shaman library}
}

@article{chowdhary2022eftsanitizer,
      author        = {Sangeeta Chowdhary and Santosh Nagarakatte},
      title         = {Fast shadow execution for debugging numerical errors using error free transformations},
      journal       = {Proceedings of the ACM on Programming Languages},
      volume        = {6},
      number        = {OOPSLA2},
      pages         = {1845--1872},
      year          = {2022},
      doi           = {10.1145/3563353},
      note          = {Presents EFTSanitizer}
}

@misc{he2026accurateresidues,
      author        = {Yumeng He and Pavel Panchekha},
      title         = {Accurate residues for floating-point debugging},
      year          = {2026},
      eprint        = {2604.06258},
      archiveprefix = {arXiv},
      url           = {https://arxiv.org/abs/2604.06258}
}

@inproceedings{benz2012fpdebug,
      author        = {Florian Benz and Andreas Hildebrandt and Sebastian Hack},
      title         = {A dynamic program analysis to find floating-point accuracy problems},
      booktitle     = {Proceedings of the 33rd ACM SIGPLAN Conference on Programming Language Design and Implementation},
      series        = {PLDI '12},
      pages         = {453--462},
      publisher     = {Association for Computing Machinery},
      address       = {New York, NY, USA},
      year          = {2012},
      doi           = {10.1145/2254064.2254118},
      note          = {Presents FpDebug}
}

@inproceedings{chowdhary2020fpsanitizer,
      author        = {Sangeeta Chowdhary and Jay P. Lim and Santosh Nagarakatte},
      title         = {Debugging and detecting numerical errors in computation with posits},
      booktitle     = {Proceedings of the 41st ACM SIGPLAN Conference on Programming Language Design and Implementation},
      series        = {PLDI 2020},
      pages         = {731--746},
      publisher     = {Association for Computing Machinery},
      address       = {New York, NY, USA},
      year          = {2020},
      doi           = {10.1145/3385412.3386004},
      note          = {Presents PositDebug and the floating-point tool FPSanitizer}
}

@inproceedings{chowdhary2021pfpsanitizer,
      author        = {Sangeeta Chowdhary and Santosh Nagarakatte},
      title         = {Parallel shadow execution to accelerate the debugging of numerical errors},
      booktitle     = {Proceedings of the 29th ACM Joint Meeting on European Software Engineering Conference and Symposium on the Foundations of Software Engineering},
      series        = {ESEC/FSE 2021},
      pages         = {615--626},
      publisher     = {Association for Computing Machinery},
      address       = {New York, NY, USA},
      year          = {2021},
      doi           = {10.1145/3468264.3468585},
      note          = {Presents PFPSanitizer}
}

@inproceedings{courbet2021nsan,
      author        = {Cl{\'e}ment Courbet},
      title         = {{NSan}: a floating-point numerical sanitizer},
      booktitle     = {Proceedings of the 30th ACM SIGPLAN International Conference on Compiler Construction},
      series        = {CC 2021},
      pages         = {83--93},
      publisher     = {Association for Computing Machinery},
      address       = {New York, NY, USA},
      year          = {2021},
      doi           = {10.1145/3446804.3446848}
}

@inproceedings{sanchezstern2018herbgrind,
      author        = {Alex Sanchez-Stern and Pavel Panchekha and Sorin Lerner and Zachary Tatlock},
      title         = {Finding root causes of floating point error},
      booktitle     = {Proceedings of the 39th ACM SIGPLAN Conference on Programming Language Design and Implementation},
      series        = {PLDI 2018},
      pages         = {256--269},
      publisher     = {Association for Computing Machinery},
      address       = {New York, NY, USA},
      year          = {2018},
      doi           = {10.1145/3192366.3192411},
      note          = {Presents Herbgrind}
}

@inproceedings{lam2016shval,
      author        = {Michael O. Lam and Barry L. Rountree},
      title         = {Floating-point shadow value analysis},
      booktitle     = {2016 5th Workshop on Extreme-Scale Programming Tools ({ESPT})},
      pages         = {18--25},
      publisher     = {IEEE},
      year          = {2016},
      doi           = {10.1109/ESPT.2016.007},
      note          = {Presents the SHVAL framework}
}

@article{jezequel2008cadna,
      author        = {Fabienne J{\'e}z{\'e}quel and Jean-Marie Chesneaux},
      title         = {{CADNA}: a library for estimating round-off error propagation},
      journal       = {Computer Physics Communications},
      volume        = {178},
      number        = {12},
      pages         = {933--955},
      year          = {2008},
      doi           = {10.1016/j.cpc.2008.02.003}
}

@inproceedings{denis2016verificarlo,
      author        = {Christophe Denis and Pablo de Oliveira Castro and Eric Petit},
      title         = {{Verificarlo}: checking floating point accuracy through {Monte Carlo} arithmetic},
      booktitle     = {2016 IEEE 23rd Symposium on Computer Arithmetic ({ARITH})},
      pages         = {55--62},
      publisher     = {IEEE},
      year          = {2016},
      doi           = {10.1109/ARITH.2016.31}
}

@misc{fevotte2016verrou,
      author        = {Fran{\c{c}}ois F{\'e}votte and Bruno Lathuili{\`e}re},
      title         = {{VERROU}: assessing floating-point accuracy without recompiling},
      year          = {2016},
      month         = oct,
      howpublished  = {HAL working paper},
      note          = {HAL identifier hal-01383417},
      url           = {https://hal.science/hal-01383417}
}

@article{graillat2011sam,
      author        = {Stef Graillat and Fabienne J{\'e}z{\'e}quel and Shiyue Wang and Yuxiang Zhu},
      title         = {Stochastic arithmetic in multiprecision},
      journal       = {Mathematics in Computer Science},
      volume        = {5},
      number        = {4},
      pages         = {359--375},
      year          = {2011},
      doi           = {10.1007/s11786-011-0103-4},
      note          = {Presents the SAM library}
}

@inproceedings{menon2018adapt,
      author        = {Harshitha Menon and Michael O. Lam and Daniel Osei-Kuffuor and Markus Schordan and Scott Lloyd and Kathryn M. Mohror and Jeffrey Hittinger},
      title         = {{ADAPT}: algorithmic differentiation applied to floating-point precision tuning},
      booktitle     = {Proceedings of the International Conference for High Performance Computing, Networking, Storage, and Analysis},
      series        = {SC '18},
      pages         = {48:1--48:13},
      publisher     = {IEEE Press},
      year          = {2018},
      doi           = {10.1109/SC.2018.00051}
}

@inproceedings{singh2023cheffp,
      author        = {Garima Singh and Baidyanath Kundu and Harshitha Menon and Alexander Penev and David J. Lange and Vassil Vassilev},
      title         = {Fast and automatic floating point error analysis with {CHEF-FP}},
      booktitle     = {2023 IEEE International Parallel and Distributed Processing Symposium ({IPDPS})},
      pages         = {1018--1028},
      publisher     = {IEEE},
      year          = {2023},
      doi           = {10.1109/IPDPS54959.2023.00105}
}

@article{zou2020atomu,
      author        = {Daming Zou and Muhan Zeng and Yingfei Xiong and Zhoulai Fu and Lu Zhang and Zhendong Su},
      title         = {Detecting floating-point errors via atomic conditions},
      journal       = {Proceedings of the ACM on Programming Languages},
      volume        = {4},
      number        = {POPL},
      articleno     = {60},
      pages         = {60:1--60:27},
      year          = {2020},
      doi           = {10.1145/3371128},
      note          = {Presents the ATOMU tool}
}

@article{yi2024fpcc,
      author        = {Xin Yi and Hengbiao Yu and Liqian Chen and Xiaoguang Mao and Ji Wang},
      title         = {{FPCC}: detecting floating-point errors via chain conditions},
      journal       = {Proceedings of the ACM on Programming Languages},
      volume        = {8},
      number        = {OOPSLA2},
      pages         = {1504--1531},
      year          = {2024},
      doi           = {10.1145/3689764}
}

@misc{baratta2025dolfinx,
      author        = {Baratta, Igor A. and Dean, Joseph P. and Dokken, J\o{}rgen S. and Habera, Michal and Hale, Jack S. and Richardson, Chris N. and Rognes, Marie E. and Scroggs, Matthew W. and Sime, Nathan and Wells, Garth N.},
      title         = {{DOLFINx}: the next generation {FEniCS} problem solving environment},
      month         = dec,
      year          = 2025,
      publisher     = {Zenodo},
      doi           = {10.5281/zenodo.18101307},
      url           = {https://doi.org/10.5281/zenodo.18101307}
}

@misc{zilian2026dolfiny,
      author        = {Zilian, Andreas and Habera, Michal and K{\"u}hner, Paul T.},
      title         = {{dolfiny}: high-level and convenience wrappers for {DOLFINx}},
      year          = {2026},
      howpublished  = {\url{https://github.com/fenics-dolfiny/dolfiny}},
      note          = {Version 0.12.0.dev0, accessed 2026-07-21}
}

@misc{habera2026rea,
      doi           = {10.5281/zenodo.23020888},
      url           = {https://zenodo.org/doi/10.5281/zenodo.23020888},
      author        = {Habera, Michal and K{\"u}hner, Paul T., and Croci, Matteo and Zilian, Andreas},
      title         = {Running Error Analysis library},
      publisher     = {Zenodo},
      year          = {2026},
      copyright     = {MIT License}
}

@misc{fenics2026ffcx,
      author        = {{FEniCS}},
      title         = {{FFCx}: the {FEniCSx} form compiler},
      year          = {2026},
      howpublished  = {\url{https://github.com/FEniCS/ffcx}},
      note          = {Version 0.12.0.dev0, accessed 2026-07-21}
}

@misc{kuhner2026ffcxbackends,
      doi           = {10.5281/zenodo.23034916},
      url           = {https://zenodo.org/doi/10.5281/zenodo.23034916},
      author        = {K{\"u}hner, Paul T., and Habera, Michal and Pachev, Benjamin and Andreas, Zilian},
      title         = {FFCx-backends},
      publisher     = {Zenodo},
      year          = {2026},
      copyright     = {GNU Lesser General Public License v3.0 only}
}

@misc{fptalks2026community,
      author        = {{FPTalks}},
      title         = {{FPTalks} community: floating-point tools and resources},
      year          = {2026},
      howpublished  = {\url{https://fptalks.org/community.html}},
      note          = {Accessed 2026-08-10}
}

@article{pence2014neohookean,
      title         = {On compressible versions of the incompressible neo-{Hookean} material},
      volume        = {20},
      issn          = {1741-3028},
      url           = {http://dx.doi.org/10.1177/1081286514544258},
      doi           = {10.1177/1081286514544258},
      number        = {2},
      journal       = {Mathematics and Mechanics of Solids},
      publisher     = {SAGE Publications},
      author        = {Pence,  Thomas J and Gou,  Kun},
      year          = {2015},
      month         = feb,
      pages         = {157--182}
}

@misc{habera2026automated,
      title         = {Automated dimensional analysis for {PDEs}},
      author        = {Habera, Michal and Zilian, Andreas},
      year          = {2026},
      eprint        = {2601.06535},
      archiveprefix = {arXiv},
      primaryclass  = {cs.MS},
      doi           = {10.48550/arXiv.2601.06535}
}

@article{dekker1971,
      author        = {Dekker, Theodorus J.},
      title         = {A floating-point technique for extending the available precision},
      journal       = {Numerische Mathematik},
      volume        = {18},
      pages         = {224--242},
      year          = {1971},
      doi           = {10.1007/BF01397083}
}

@article{ogita2005,
      author        = {Ogita, Takeshi and Rump, Siegfried M. and Oishi, Shin'ichi},
      title         = {Accurate sum and dot product},
      journal       = {SIAM Journal on Scientific Computing},
      volume        = {26},
      number        = {6},
      pages         = {1955--1988},
      year          = {2005},
      doi           = {10.1137/030601818}
}

@book{trefethen1997,
      author        = {Trefethen, Lloyd N. and Bau, David, III},
      title         = {Numerical Linear Algebra},
      publisher     = {Society for Industrial and Applied Mathematics},
      address       = {Philadelphia, PA},
      year          = {1997},
      isbn          = {978-0-89871-361-9}
}

@book{brenner2008fem,
      author        = {Susanne C. Brenner and L. Ridgway Scott},
      title         = {The Mathematical Theory of Finite Element Methods},
      edition       = {3},
      series        = {Texts in Applied Mathematics},
      volume        = {15},
      publisher     = {Springer},
      address       = {New York},
      year          = {2008},
      doi           = {10.1007/978-0-387-75934-0}
}

@article{lange2020faithfully,
      author        = {Marko Lange and Siegfried M. Rump},
      title         = {Faithfully rounded floating-point computations},
      journal       = {ACM Transactions on Mathematical Software},
      volume        = {46},
      number        = {3},
      articleno     = {21},
      numpages      = {20},
      year          = {2020},
      doi           = {10.1145/3290955}
}

@book{moore2009interval,
      author        = {Ramon E. Moore and R. Baker Kearfott and Michael J. Cloud},
      title         = {Introduction to Interval Analysis},
      publisher     = {Society for Industrial and Applied Mathematics},
      address       = {Philadelphia, PA},
      year          = {2009},
      isbn          = {978-0-89871-669-6},
      doi           = {10.1137/1.9780898717716}
}

@book{muller2018handbook,
      author        = {Jean-Michel Muller and Nicolas Brunie and Florent de Dinechin and Claude-Pierre Jeannerod and Mioara Joldes and Vincent Lef{\`e}vre and Guillaume Melquiond and Nathalie Revol and Serge Torres},
      title         = {Handbook of Floating-Point Arithmetic},
      edition       = {2},
      publisher     = {Birkh{\"a}user},
      address       = {Cham},
      year          = {2018},
      doi           = {10.1007/978-3-319-76526-6},
      isbn          = {978-3-319-76525-9}
}

@manual{aapcs64,
      author        = {{Arm Limited}},
      title         = {Procedure call standard for the {Arm} 64-bit architecture ({AArch64})},
      organization  = {Arm Limited},
      year          = {2026},
      month         = jan,
      note          = {AAPCS64, release 2025Q4, Section 10.1.1, Table 3},
      url           = {https://github.com/ARM-software/abi-aa/blob/main/aapcs64/aapcs64.rst},
      urldate       = {2026-08-17}
}

@software{cppjit,
      author        = {{compiler-research}},
      title         = {{CppJIT}: fast and automatic {Python}--{C++} interoperability},
      url           = {https://github.com/compiler-research/cppjit},
      year          = {2026},
      note          = {GitHub repository}
}

@misc{nanobind,
      author        = {Jakob Wenzel},
      year          = {2022},
      note          = {https://github.com/wjakob/nanobind},
      title         = {{nanobind}: tiny and efficient {C++}/{Python} bindings}
}

@misc{nvidiarubin2026,
      author        = {{NVIDIA Corporation}},
      title         = {{NVIDIA Vera Rubin NVL72}},
      year          = {2026},
      url           = {https://www.nvidia.com/en-us/data-center/vera-rubin-nvl72/},
      note          = {Accessed: 2026-09-07}
}
